\documentclass[reqno]{amsart}
\usepackage{amsthm}
\usepackage{amssymb}
\usepackage{graphicx}
\usepackage{enumerate}
\usepackage{cite}
\usepackage{amsrefs}

\theoremstyle{plain}
\newtheorem{thm}{Theorem}

\newtheorem*{thm*}{Theorem}
\newtheorem{lem}[thm]{Lemma}
\newtheorem{cor}[thm]{Corollary}

\newcommand{\figdir}[0]{.}
\newcommand{\gfextn}[0]{eps}
\newcommand{\rheinwidth}[0]{1.3in}
\newcommand{\longbox}[2]{\framebox[51pt]{\phantom{0}\texttt{#1}\ldots\texttt{#2}\phantom{0}}}
\newcommand{\shortbox}[0]{\framebox[35pt]{\phantom{\texttt{+}}0\phantom{\texttt{+}}}}

\newcommand{\tinybox}[0]{\framebox[8pt]{\phantom{0}}}
\newcommand{\abs}[1]{\lvert#1\rvert}
\newcommand\fixspacetop{\rule{0pt}{2.2ex}}

\numberwithin{thm}{section}
\numberwithin{equation}{section}

\subjclass[2010]{Primary: 52B60; Secondary: 11R09, 52A10, 52B05}
\keywords{Reinhardt polygon, Reinhardt polynomial, dihedral composition.}

\begin{document}

\title{Most Reinhardt polygons are sporadic}
\date{\today}

\author{Kevin G. Hare}
\address{Department of Pure Mathematics, University of Waterloo, Waterloo, Ontario, Canada N2L 3G1.}
\email{kghare@uwaterloo.ca}
\thanks{Research of K.G. Hare was partially supported by NSERC}

\author{Michael J. Mossinghoff}
\address{Department of Mathematics and Computer Science, Davidson College, Davidson, NC 28036 USA.}
\email{mimossinghoff@davidson.edu}
\thanks{This work was partially supported by a grant from the Simons Foundation (\#210069 to Michael Mossinghoff).}

\begin{abstract}
A \textit{Reinhardt polygon} is a convex $n$-gon that, for $n$ not a power of $2$, is optimal in three different geometric optimization problems, for example, it has maximal perimeter relative to its diameter.
Some such polygons exhibit a particular periodic structure; others are termed \textit{sporadic}.
Prior work has described the periodic case completely, and has shown that sporadic Reinhardt polygons occur for all $n$ of the form $n=pqr$ with $p$ and $q$ distinct odd primes and $r\geq2$.
We show that (dihedral equivalence classes of) sporadic Reinhardt polygons outnumber the periodic ones for almost all $n$, and find that this first occurs at $n=105$.
We also determine a formula for the number of sporadic Reinhardt polygons when $n=2pq$ with $p$ and $q$ distinct odd primes.
\end{abstract}

\maketitle

\section{Introduction}\label{secIntroduction}

Reinhardt polygons are a class of convex polygons that are optimal in three different geometric optimization problems: they have maximal perimeter and maximal width with respect to their diameter, and maximal width relative to their perimeter \citelist{\cite{Reinhardt22}\cite{BezdekFodor00}\cite{AudetHansenMessine09}}.
Further, for $n$ not a power of $2$, the Reinhardt polygons with $n$ sides are precisely the optimal convex polygons in these three problems.  
We describe Reinhardt polygons only briefly here; for more details we direct the reader to \citelist{\cite{Mossinghoff06AMM}\cite{Mossinghoff06DCG}\cite{HareMossinghoff13}}.

A \textit{Reinhardt polygon} is an equilateral convex polygon $P$ that may be inscribed in a Reuleaux polygon $R$ in such a way that every vertex of $R$ is a vertex of $P$.
If we define the \textit{skeleton} of a polygon to be its vertices, together with the line segments that connect vertices at maximal distance from one another, then it follows that the skeleton of a Reinhardt polygon with $n$ sides contains a star polygon, where each interior angle has measure an integer multiple of $\pi/n$.
For example, when $n$ is odd, the regular $n$-gon is a Reuleaux polygon, and its associated skeleton is the regular star polygon with $n$ points.
A Reinhardt polygon can be described by naming the measures of each of the angles of its associated star polygon, in the order of their visitation as one traverses it edges.
Since each angle is $k_i\pi/n$ for some integer $k_i$, and their sum is $\pi$, it follows that we may describe a Reuleaux polygon as a composition of $n$ into an odd number of parts $\ell$: $[k_1,\ldots,k_\ell]$.
Further, since we consider two polygons to be equivalent if one can be obtained from the other by some combination of rotations and flips, we associate a Reinhardt polygon with an equivalence class of such compositions under a dihedral action.
Such an equivalence class is called a \textit{dihedral composition}.

Figure~\ref{figReinhardt30} exhibits six Reinhardt polygons with $n=30$ sides, along with their associated dihedral compositions.
The skeleton of each polygon is exhibited in the network of line segments within it, and the subset of each skeleton forming the associated star polygon is depicted with thicker lines.
The first row of the polygons here have an evident periodic structure: for these, the associated composition is a composition for some divisor $m$ of $n$ into an odd number of parts, repeated $n/m$ times, and the polygon exhibits a corresponding rotational symmetry.
In these cases we abbreviate the composition by writing just the composition of $m$, together with $n/m$ as an exponent to show the number of times the sequence is repeated.
Reinhardt polygons with this property are called \textit{periodic Reinhardt polygons}, and these are well understood.
In \cite{Mossinghoff11}, a formula for the number of such polygons is derived (and independently in \cite{Gashkov07} for the case of cyclic equivalence classes).
If we let $E_0(n)$ denote the number of periodic Reinhardt polygons with $n$ sides under dihedral equivalence classes, then it follows from this formula that
\begin{equation}\label{eqnPeriodic}
E_0(n) = \frac{p2^{n/p}}{4n}\left(1+o(1)\right),
\end{equation}
for positive integers $n$ with smallest odd prime divisor $p$.
We remark that there are several additional periodic Reinhardt $30$-gons, beyond the three shown in Figure~\ref{figReinhardt30}, since $E_0(30)=38$.

\begin{figure}[tbh]
\caption{Some Reinhardt polygons with $n=30$ sides.}\label{figReinhardt30}
\begin{tabular}{ccc}
\includegraphics[width=\rheinwidth]{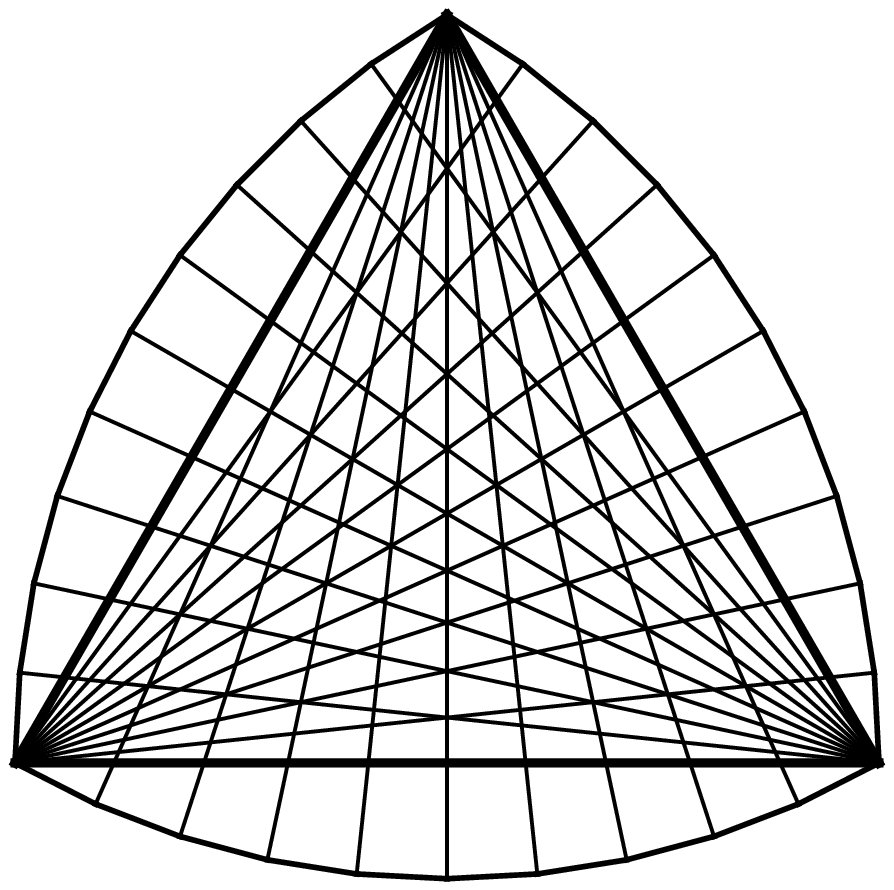} &
\includegraphics[width=\rheinwidth]{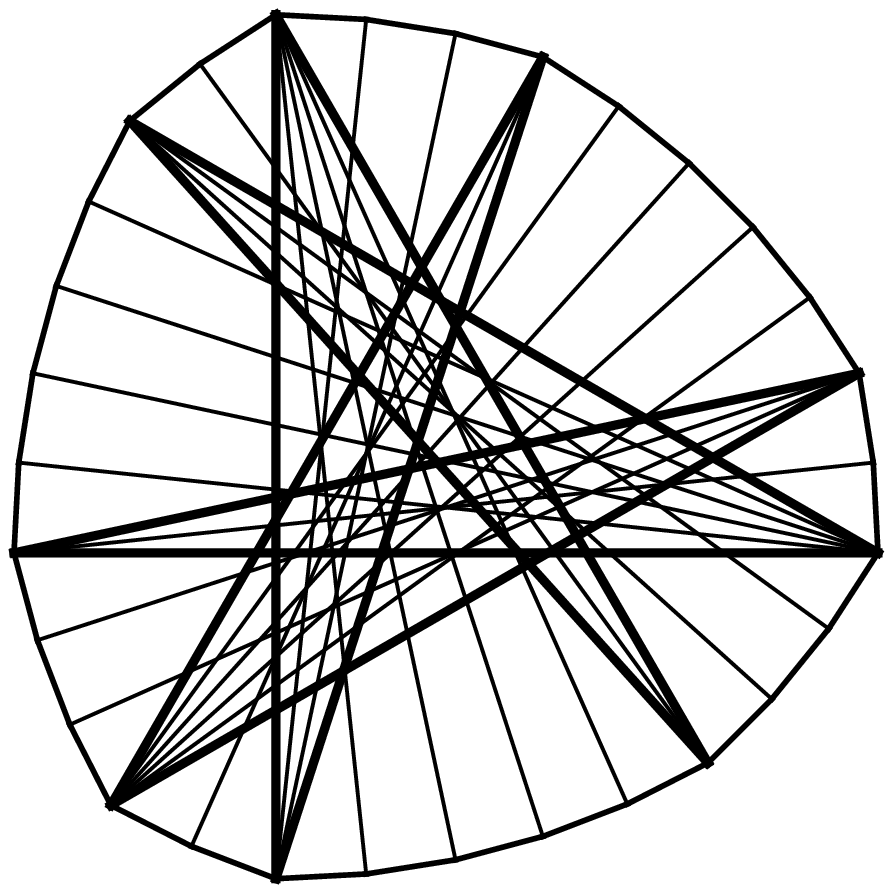} &
\includegraphics[width=\rheinwidth]{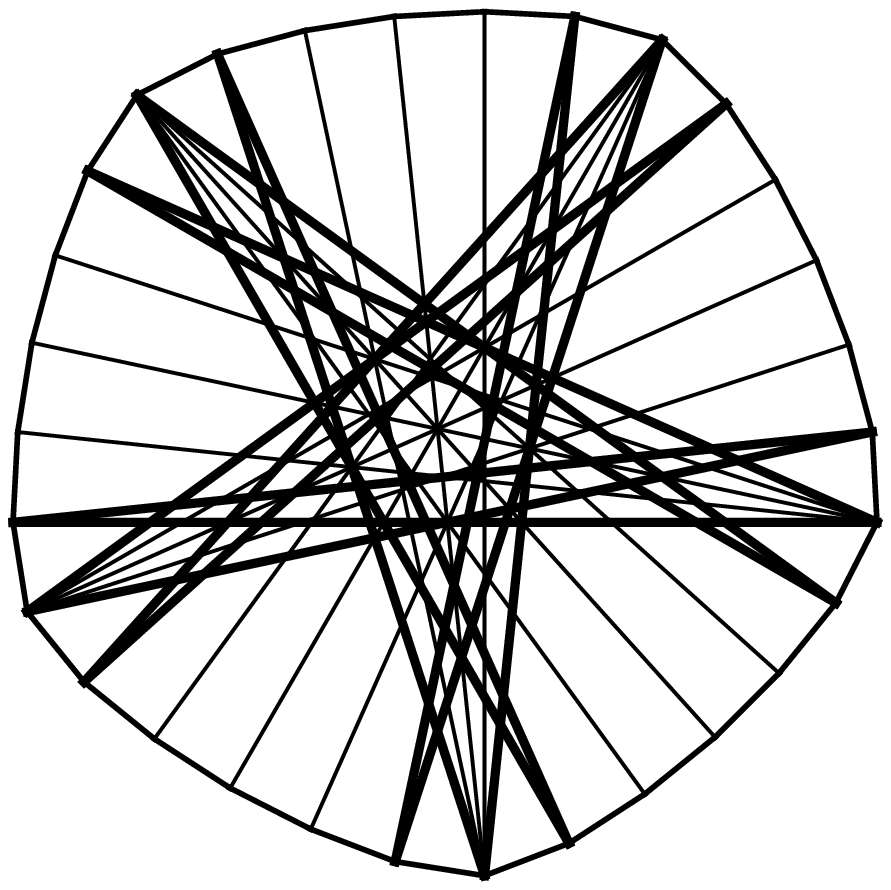}\\
\small $[(10)^3]$ &
\small $[(5,3,2)^3]$ &
\small $[(4,1,1)^5]$ \\[12pt]
\includegraphics[width=\rheinwidth]{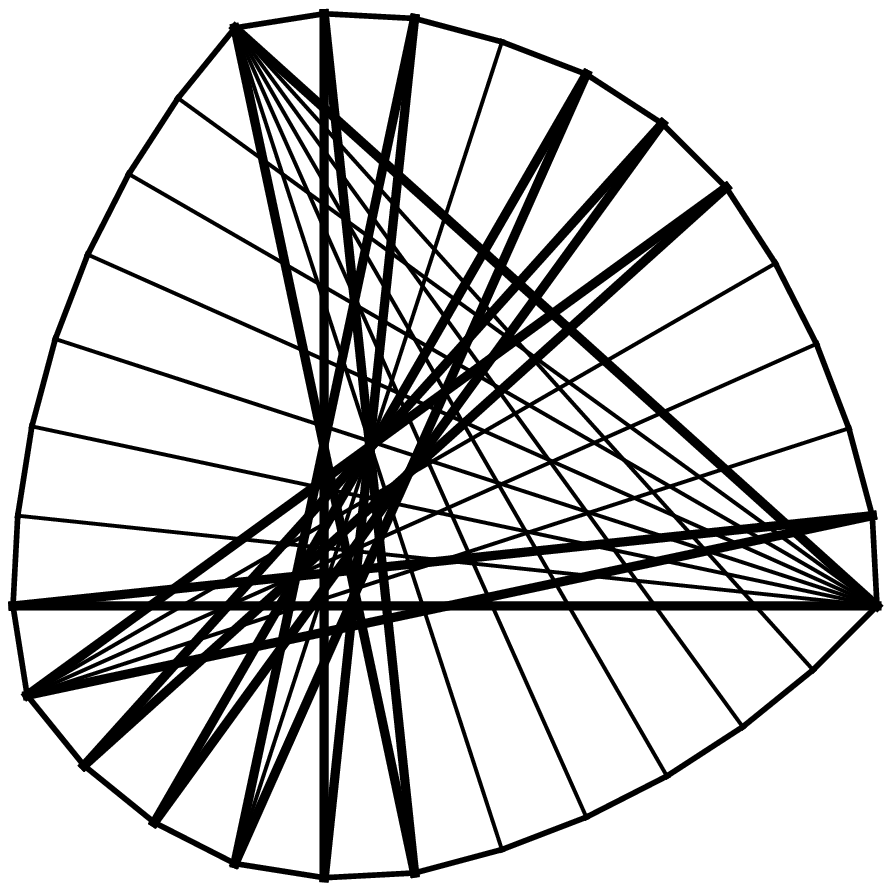} &
\includegraphics[width=\rheinwidth]{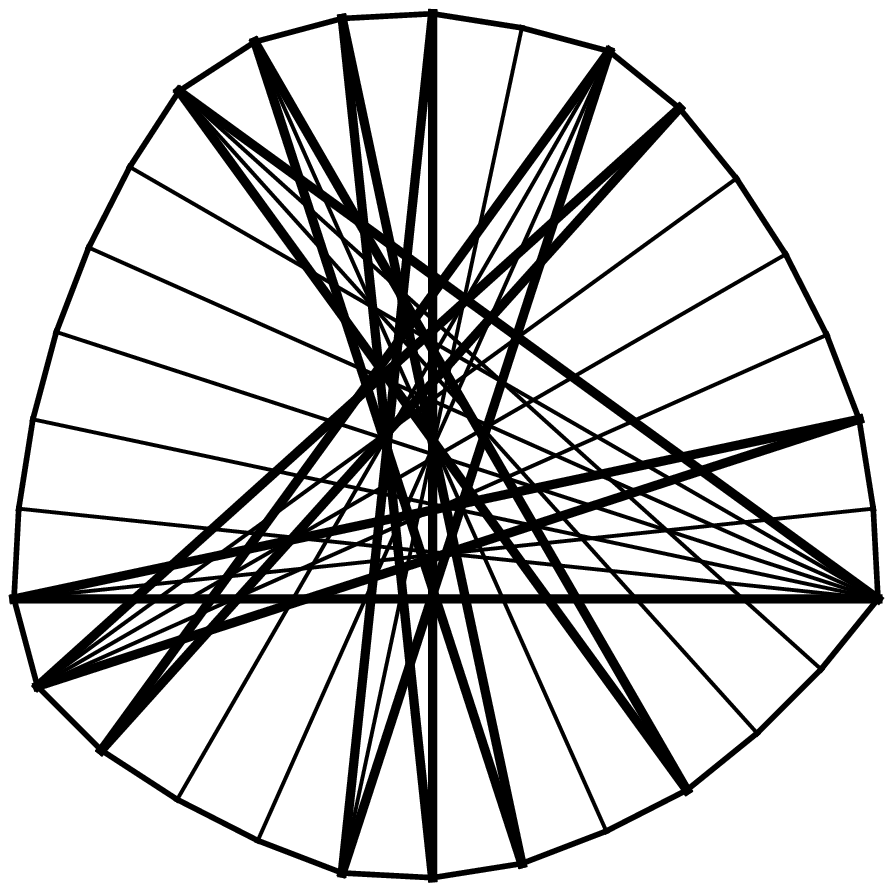} &
\includegraphics[width=\rheinwidth]{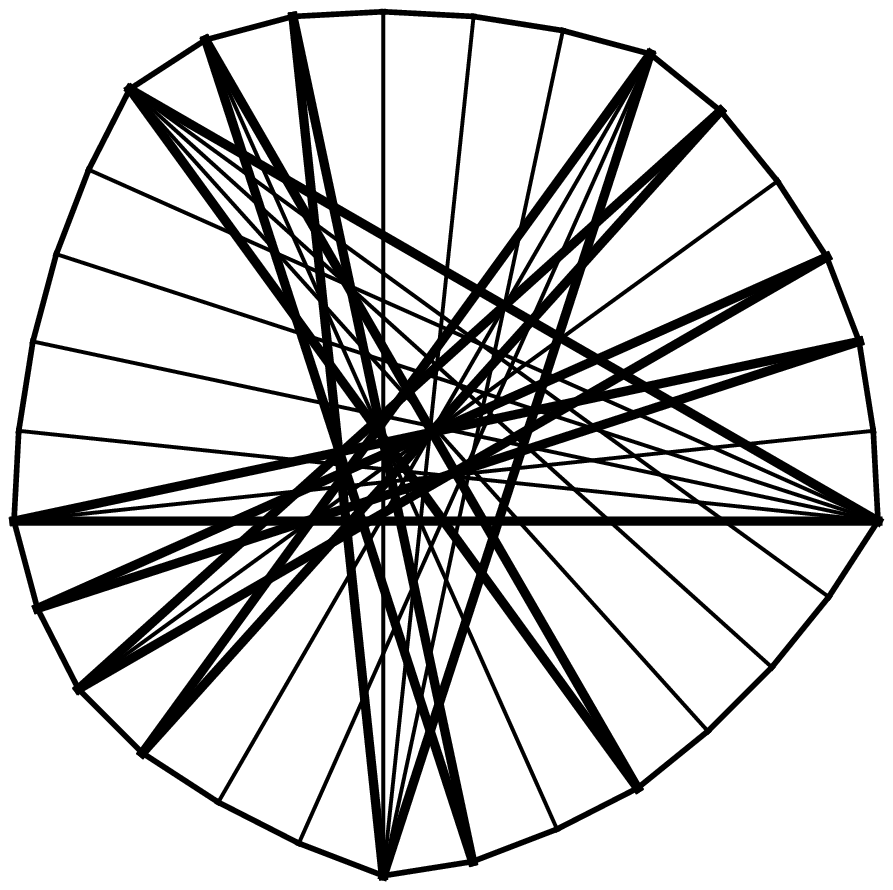}\\
\small $[7,6,1,1,1,1,2,$ &
\small $[6,3,1,2,1,1,1,$ &
\small $[5,4,1,2,1,1,4,$\\
\small \quad $\;1,1,1,1,1,4,1,1]$ &
\small \quad $\;1,2,3,1,1,4,1,2]$ &
\small \quad $\;3,1,1,2,1,1,1,2]$
\end{tabular}
\end{figure}

Reinhardt \cite{Reinhardt22} established a correspondence between these polygons and certain polynomials.
We define a \textit{Reinhardt polynomial} for $n$ to be a polynomial $F(z)$ having $\deg(F)<n$, all coefficients in $\{-1,0,1\}$, its nonzero coefficients alternate in sign, there are an odd number of nonzero coefficients, and the cyclotomic polynomial $\Phi_{2n}(z)$ is a factor of $F(z)$.
Reinhardt in fact also required $F(0)=1$ (and we employed this stricter definition in \cite{HareMossinghoff13}), but it is convenient here to broaden the definition slightly.
A Reinhardt polynomial $F(z)=\pm\sum_{i=0}^\ell (-1)^i z^{k_i}$ corresponds to the dihedral composition $[k_1-k_0,k_2-k_1,\ldots,k_\ell-k_{\ell-1},n-k_\ell+k_0]$ of $n$ into an odd number of parts; the condition that $\Phi_{2n}(z)\mid F(z)$ guarantees that this composition produces a closed path when constructing the star polygon.
For details, see \cite{Mossinghoff06AMM}.

We say a Reinhardt polygon is \textit{sporadic} if it is not periodic, and define $E_1(n)$ to be the number of sporadic Reinhardt $n$-gons, again using dihedral equivalence classes.
For example, each of the polygons in the second row of Figure~\ref{figReinhardt30} is sporadic for $n=30$.
In fact, these are the only three sporadic Reinhardt triacontagons, so $E_1(30)=3$.

In \cite{Mossinghoff11}, the second author proved that $E_1(2^a p^b)=0$ if $p$ is an odd prime, $a\geq0$, and $b\geq1$, and computed $E_1(n)$ for a number of other values of $n$, in fact for all positive integers having $n-\varphi(2n)\leq46$.
Aided by that data, it was speculated there that $E_1(pq)=0$ and that $E_1(pqr)>0$ for distinct odd primes $p$ and $q$ and $r\geq2$.
Despite the fact that $E_1(n)<E_0(n)$ in each case computed there, it was also conjectured that $E_1(n)>E_0(n)$ for almost all positive integers $n$, and that this would first occur at $n=105$.
In \cite{HareMossinghoff13}, the authors proved the first two of these conjectures, regarding the cases $n=pq$ and $n=pqr$, and they further showed by explicit construction that a positive proportion of the Reinhardt $n$-gons are sporadic in the case $n=pqr$, when $p$ and $q$ are fixed and $r$ is large.
Some additional empirical evidence that the sporadic polygons may outnumber the periodic ones at $n=105$ was also presented.

In this article, we describe a generalized method for constructing sporadic Reinhardt $n$-gons for integers of the form $n=pqr$ with $p$ and $q$ odd primes and $r\geq2$.
Our method allows us to prove that $E_1(n)>E_0(n)$ for almost all positive integers $n$, and to show that this first occurs at $n=105$.
We prove the following theorems.

\begin{thm}\label{thmSporadicsWin}
For almost all positive integers $n$, the number of sporadic Reinhardt polygons exceeds the number of periodic Reinhardt polygons.
\end{thm}

\begin{thm}\label{thm105}
The smallest integer $n$ for which the number of sporadic Reinhardt polygons exceeds the number of periodic ones is $n=105$.
\end{thm}

We also establish an exact formula for $E_1(n)$ for integers of the form $n=2pq$, where $p$ and $q$ are distinct odd primes.
For this, recall that the \textit{Fermat quotient} of a prime $p$ with base $a$ having $p\nmid a$ is the integer $(a^{p-1}-1)/p$.

\begin{thm}\label{thm2pq}
For $p$ and $q$ distinct odd primes, the number of sporadic Reinhardt polygons with $n=2pq$ sides is the product of the Fermat quotients with base $2$ for $p$ and $q$:
\[
E_1(2pq) = \frac{(2^{p-1}-1)(2^{q-1}-1)}{pq}.
\]
\end{thm}

Section~\ref{secConstruction} describes and verifies our new method for constructing Reinhardt polynomials for integers of the form $n=pqr$.
This method depends on the choice of a composition of $r$ into an even number of parts.
Section~\ref{secCounting} determines the number of different Reinhardt polynomials that may be produced by our method for such a fixed composition, studies how many of these correspond to sporadic Reinhardt polygons, and establishes Theorem~\ref{thmSporadicsWin}.
Section~\ref{secComputations} describes our implementation of this construction and the results of computations performed for a number of values of $n$, including $n=105$ for establishing Theorem~\ref{thm105}.
Section~\ref{sec2pq} derives the formula for the case $n=2pq$.

\section{Constructing Reinhardt polynomials}\label{secConstruction}

Let $p$ and $q$ be distinct odd primes, and let $r$ be an integer with $r\geq2$.
We describe a method for constructing a large family of Reinhardt polynomials for $n=pqr$.
As in \cite{HareMossinghoff13}, by using the result of de Bruijn \cite{deBruijn53}, we know that the principal ideal generated by the cyclotomic polynomial $\Phi_{2n}(z)$ in $\mathbb{Z}[z]$ is generated by $z^n+1$ and the collection $\{\Phi_p(-z^{n/p}) : \textrm{$p$ is an odd prime and $p\mid n$}\}$.
Thus, it suffices to construct polynomials $f_1(z)$ and $f_2(z)$ so that
\[
F(z) = f_1(z)\Phi_q(-z^{pr}) + f_2(z)\Phi_p(-z^{qr})
\]
has an odd number of nonzero coefficients, each $\pm1$, which alternate in sign.
Our method developed here generalizes the construction described in \cite{HareMossinghoff13}.

\subsection{Description}\label{subsecDescription}
We first choose a composition $\textbf{c}=(r_1,\ldots,r_{2m})$ of $r$ into an even number of parts $2m$, so each $r_i$ is a positive integer and $\sum_{i=1}^{2m} r_i = r$.
We construct $f_1(z)$ as a polynomial of degree at most $pr-1$ by describing its coefficients in $p$ blocks of size $r$, which we label $A_1$, \ldots, $A_p$.
Each block $A_i$ is then subdivided into $2m$ sub-blocks, denoted $A_{i,j}$ with $1\leq j\leq2m$, according to the selected composition $\textbf{c}$.
In order to describe these components, we require some notation.

For a nonnegative integer $k$ and $b\in\{1,-1\}$, let $S_o(k,b)$ denote the set of sequences of length $k$ over $\{-1,0,1\}$ having an odd number of nonzero terms, the first of which is $b$, and which alternate in sign.
For example, \texttt{+0-0+}, \texttt{+-+-+}, and \texttt{000+0} are all members of $S_o(5,1)$.
Likewise, let $S_e(k,b)$ denote the set of sequences of length $k$ over $\{-1,0,1\}$ with an even number (possibly zero) of nonzero terms, the first of which is $b$ (if there is a nonzero term) and alternating in sign.
It is worth noting that $\texttt{00000}$ is in both $S_e(5,1)$ and $S_e(5,-1)$.
Finally, let $Z(k)$ denote the singleton set containing only the sequence of length $k$ consisting entirely of zeros.

We may now describe the construction of the sub-blocks $A_{i,j}$.
Choose a fixed value for $s$ from $\{-1,1\}$.
Then select
\begin{equation}\label{eqnAs}
A_{i,j} \in \begin{cases}
S_o(r_1+1, (-1)^{i+1}s), & j=1,\\
S_e(r_j+1, (-1)^is), & \textrm{$j\geq3$ odd},\\
Z(r_j-1), & \textrm{$j$ even}.
\end{cases}
\end{equation}
As a special case, we require that $A_{1,1}$ begins with $s$, not $0$.
The sequence of coefficients for $f_1(z)$ is created by juxtaposing these $p$ blocks: $A_1 A_2 \ldots A_p$.

The polynomial $f_2(z)$ is also constructed relative to the composition $\textbf{c}$ of $r$.
In this case, we describe $q$ blocks $B_1$, \ldots, $B_q$, each with size $r$.
In the same way, each block $B_i$ is composed of $2m$ sub-blocks, denoted $B_{i,j}$ with $1\leq j\leq2m$, and in this case we select
\begin{equation}\label{eqnBs}
B_{i,j} \in \begin{cases}
Z(r_j-1), & \textrm{$j$ odd},\\
S_e(r_j+1, (-1)^is), & \textrm{$j$ even}.
\end{cases}
\end{equation}
Then we construct the sequence of coefficients for $f_2(z)$ by shifting the terms in the aggregated blocks by one position to the right, with the last element negated and then moved to the first position.
If $R$ denotes this shift-right-and-negate operator, then the coefficients of $f_2(z)$ are given in sequence by $R(B_1 B_2\ldots B_q)$.

\subsection{Example}\label{subsecExample}
Suppose $n=120$, with $p=3$, $q=5$, and $r=8$, and we choose the composition $\mathbf{c}=(1,3,2,2)$, and $s=1$.
We must select $A_{1,1}=$ \texttt{+0}, then $A_{1,2}$ must be \texttt{00}.
There are four different possibilities for $A_{1,3}$: \texttt{000}, \texttt{-+0}, \texttt{-0+}, or \texttt{0-+}, and then $A_{1,4}$ must be \texttt{0}.
Suppose we select $A_{1,3}=$ \texttt{-+0}.
Then $A_1=$ \texttt{+0|00|-+0|0}, where vertical bars have been inserted between sub-blocks to aid in parsing.
For $A_2$, the initial signs in the odd-indexed sub-blocks are inverted, so here we may choose $A_2=$ \texttt{0-|00|+0-|0}, and then for $A_3$ we may select \texttt{0+|00|000|0}.
This makes
\[
f_1(z) = 1-z^4+z^5-z^9+z^{12}-z^{14}+z^{17}.
\]
For $f_2(z)$, we see that $B_{i,1}$ is the empty sequence, $B_{i,2}\in S_e(4,(-1)^i)$, $B_{i,3}=$ \texttt{0},  and $B_{i,4}\in S_e(3,(-1)^i)$.
Suppose we choose $B_1=$ \texttt{|-+-+|0|-+0}, $B_2=$ \texttt{|+-00|0|0+-}, $B_3=$ \texttt{|0000|0|-0+}, $B_4=$ \texttt{|+00-|0|000}, and $B_5=$ \texttt{|0000|0|0-+}.
After applying the $R$ operator to the combined sequence, we construct
\[
f_2(z) = -1-z+z^2-z^3+z^4-z^6+z^7+z^9-z^{10}+z^{15}-z^{16}-z^{22}+z^{24}+z^{25}-z^{28}-z^{39}.
\]
We then compute $F(z)=f_1(z)\Phi_5(-z^{24}) + f_2(z)\Phi_3(-z^{40})$.
The first 60 coefficients of each summand are then
\begin{quote}
{\small
\texttt{+000-+00  0-00+0-0 0+000000 \underline{-000+-00  0+00-0+0 0-000000} +000-+00~0-00}\\ 
\texttt{--+-+0-+ 0+-0000+ -00000-0 ++00-000 0000000- \underline{++-+-0+- 0-+0000- +000}}  
}
\end{quote}
Here, the coefficients are listed in blocks of size $r=8$, with underlining showing where prior blocks were negated after multiplying $f_1(z)$ and $f_2(z)$ by $\Phi_q(-z^{pr})$ and $\Phi_p(-z^{pr})$ respectively.
The last 60 coefficients of each summand for $F(z)$ are
\begin{quote}
{\small
\texttt{+0-0 0+000000 \underline{-000+-00  0+00-0+0 0-000000} +000-+00  0-00+0-0 0+000000}\\  
\texttt{\underline{00+0 --00+000 0000000+} --+-+0-+ 0+-0000+ -00000-0 ++00-000 0000000-}  
}
\end{quote}
Summing both, we find the following sequence of coefficients:
\begin{quote}
{\small
\texttt{0-+-0+-+ 00-0+0-+ -+0000-0 0+000-00 0+00-0+- +0-+-0+- +-+0-+0-~+-00}\\
\texttt{+000 -000+000 -000+-0+ -0+-000+ 00-0000+ 0000-+-0 +00000-0 0+00000-}
}
\end{quote}
The polynomial $F(z)$ constructed from this list has all its coefficients in $\{-1,0,1\}$, its $53$ nonzero coefficients alternate in sign, and $\Phi_{240}(z)\mid F(z)$.
Thus, $F(z)$ is a Reinhardt polynomial for $n=120$; its corresponding dihedral composition is
\begin{align*}
[&1,1,2,1,1,3,2,2,1,1,1,5,3,4,4,3,2,1,1,2,1,1,2,1,1,1,\\
&1,2,1,2,1,1,3,4,4,4,4,1,2,1,2,1,4,3,5,5,1,1,2,6,3,6,2].
\end{align*}

\subsection{Verifying the construction}\label{subsecVerifying}
We claim that every polynomial $F(z)$ constructed using this method is a Reinhardt polynomial for $n$.
Clearly $F(z)$ has $\Phi_{2n}(z)$ as a factor.
Also, certainly $f_1(1)\equiv1$ mod $2$ and $f_2(1)\equiv0$ mod $2$, so $F(1)\equiv1$ mod $2$ since $q$ is an odd prime.
We therefore need to show that $F(z)$ has $\{-1,0,1\}$ coefficients, and that its nonzero coefficients alternate in sign.
Clearly $f_1(z)$ and $f_2(z)$ have $\{-1,0,1\}$ coefficients, so we must consider how these values interact when we construct $F(z)$.
We consider the coefficients of $F(z)$ in blocks of size $r$.
Consider the $n$th such block with $1\leq n\leq pq$, so the coefficients of $z^{(n-1)r}$, \ldots, $z^{nr-1}$, and suppose $n=i+kp$.
If $k$ is even, then the contribution to this block of coefficients from $f_1(z)\Phi_q(-z^{pr})$ is the sequence from $A_i$; if $k$ is odd, then $\overline{A_i}$, the negated sequence, is used.
Similarly, if $n=j+\ell q$ and $\ell$ is even, then for the contribution from $f_2(z)\Phi_p(-z^{qr})$, the first $r-1$ values from $B_j$ are used, along with the last value from $B_{j-1}$ (unless $j=1$, in which case it is the last value from $\overline{B_q}$), and if $\ell$ is odd then the negated values are used.
Also, since $i+kp=j+\ell q$, then it follows that $i+j$ and $k+\ell$ have the same parity since $p$ and $q$ are both odd.
Thus, if $i\equiv j$ mod $2$, then either $A_i$ is matched with (shifted) $B_j$, or $\overline{A_i}$ with $\overline{B_j}$; if $i\not\equiv j$ mod $2$ then $A_i$ is paired with (shifted) $\overline{B_j}$, or $\overline{A_i}$ with $B_j$.

Suppose $i$, $j$, $k$, and $\ell$ are all even for a particular block $n$.
We can exhibit the interaction of the sequences $A_i$ and $B_j$ in our construction for this block of size $r$ in $F(z)$ in the following diagram.
Here, \texttt{+}\ldots\texttt{+} denotes a sequence in $S_o(k,1)$ for the indicated length $k$, and \texttt{-}\ldots\texttt{+} denotes a selection from $S_e(k,-1)$ for the required length.
The short box at the beginning of the second line denotes the last element of the prior $B$ block, which we note will either be 0 or $-1$.

\begin{equation}\label{eqnOverlaps}
\begin{split}
A_i \quad&\overbrace{\longbox{+}{+}}^{r_1+1}\overbrace{\shortbox}^{r_2-1}\overbrace{\longbox{-}{+}}^{r_3+1}\overbrace{\shortbox}^{r_4-1}\overbrace{\longbox{-}{+}}^{r_5+1}\overbrace{\shortbox}^{r_6-1}\ldots\\[4pt]
B_j \quad&\,\tinybox\underbrace{\shortbox}_{r_1-1}\underbrace{\longbox{-}{+}}_{r_2+1}\underbrace{\shortbox}_{r_3-1}\underbrace{\longbox{-}{+}}_{r_4+1}\underbrace{\shortbox}_{r_5-1}\underbrace{\longbox{-}{+}}_{r_6+1}\ldots
\end{split}
\end{equation}

We see that nontrivial overlaps may occur only at the boundaries of the all-zero sub-blocks, and in every case, if two nonzero values coincide, then they have opposite sign.
Further, these cancellations cannot disturb the alternating sign pattern in the sum.
The same diagram holds when $i$, $j$, $k$, and $\ell$ are all odd, as well as when $i$ and $k$ are even and $j$ and $\ell$ are odd, and when $i$ and $k$ are odd and $j$ and $\ell$ are even.
The remaining four cases are similar: the corresponding diagram merely reverses the roles of \texttt{+} and \texttt{-}.
It follows that each block of $r$ coefficients in $F(z)$ has the required alternating sign pattern.
By considering the leading and trailing nonzero terms of neighboring blocks, and the possible cancellation that may occur at their boundaries, it follows that this construction always produces a Reinhardt polynomial.
\qed

\subsection{Symmetry of the construction}
\label{sec:symmetry}

In our construction, the roles of $p$ and $q$ are not identical.
From $p$, we construct our $A_{i,j}$, the first of which has the form
    $S_o$ and the rest of which have the form $S_e$ or $Z$.
From $q$, we construct our $B_{i,j}$, all of which have the form $S_e$ or 
    $Z$.
The next result shows that, despite this, 
    the roles of $p$ and $q$ are symmetric.
As a corollary of this, from a computational point of view, we may assume
    that $p < q$.

\begin{lem}
\label{lem:symmetry}
Let $f(z) = f_1(z) \Phi_q(-z^{pr}) + f_2(z) \Phi_p(-z^{qr})$ be a 
    Reinhardt polynomial constructed as above, where $f_1(z)$ is composed
    of blocks $A_1$, $A_2$, \ldots, $A_p$ and $f_2(z)$ is composed of blocks
    $B_1$, $B_2$, \ldots, $B_q$.
Then there exists a dihedrally equivalent 
     $f'(z) = f'_1(z) \Phi_p(-z^{qr}) + f'_2(z) \Phi_q(-z^{pr})$
    where $f'_1(z)$ is composed
    of blocks $A'_1$, $A'_2$, \ldots, $A'_q$ and $f'_2(z)$ is composed of blocks
    $B'_1$, $B'_2$, \ldots, $B'_p$.
\end{lem}

\begin{proof}
Let $\mathbf{c} = (r_1, r_2, \ldots, r_{2m})$.
Select sub-blocks $A_{i,j}$ and $B_{i,j}$ for $1\leq j\leq2m$ as in \eqref{eqnAs} and \eqref{eqnBs} to form the blocks $A_1$, \ldots, $A_p$ and $B_1$, \ldots, $B_q$.
For each $A_i$, notice that the first nonzero term is $(-1)^{i+1} s$, and that
    there are an odd number of alternating nonzero terms.
For the first term of each $A_i$, add $(-1)^is$.
That is, if the first term is $(-1)^{i+1}s$, then it is changed to $0$, and if
    it is $0$ then it is changed to $(-1)^is$.)
Call these modified blocks $\widetilde A_i$.
We see that each $\widetilde A_i$ has an even number of alternating nonzero terms.

For each $B_i$, notice that the last nonzero term, if it exists, is $(-1)^{i+1}s$, and that
    there are an even number of alternating nonzero terms.
For the last term of each $B_i$, add $(-1)^{i}s$.
That is, if the last term is $(-1)^{i+1}s$, then it is changed to $0$, and if
    it is $0$, then it is changed to $(-1)^{i}s$.
Call these modified blocks $\widetilde B_i$.
We see that these $\widetilde B_i$ have an odd number of alternating nonzero terms.

Since the changes from $B_i$ to $\widetilde B_i$ precisely balance the alterations from $A_i$ to $\widetilde A_i$, we note that $f(z) = \widetilde f(z)$.
We now have
\[
\widetilde A_{i,j} \in \begin{cases}
S_e(r_j+1, (-1)^is), & \textrm{$j$ odd},\\
Z(r_j-1), & \textrm{$j$ even}.
\end{cases}
\]
and 
\[
\widetilde B_{i,j} \in \begin{cases}
Z(r_j-1), & \textrm{$j$ odd},\\
S_e(r_j+1, (-1)^is), & \textrm{$j<2m$ even}, \\
S_o(r_j+1, (-1)^is), & \textrm{$j = 2m$}.
\end{cases}
\]

Let $\mathrm{Reverse}(v_1 v_2 \ldots v_{2m} ) = v_{2m} v_{2m-1} \ldots v_1$.
Define $A'_{i,j} = -\mathrm{Reverse}(\widetilde B_{i,2m+1-j})$ and 
        $B'_{i,j} = -\mathrm{Reverse}(\widetilde A_{i,2m+1-j})$.
Then $A'_i$ and $B'_i$ have the form
\[
A'_{i,j} \in \begin{cases}
S_o(r_{2m}+1, (-1)^{i+1}s), & j=1,\\
S_e(r_{2m+1-j}+1, (-1)^is), & \textrm{$j\geq3$ odd},\\
Z(r_{2m+1-j}-1), & \textrm{$j$ even},
\end{cases}
\]
and 
\[
B'_{i,j} \in \begin{cases}
Z(r_{2m+1-j}-1), & \textrm{$j$ odd},\\
S_e(r_{2m+1-j}+1, (-1)^is), & \textrm{$j$ even},
\end{cases}
\]
as required. 
Here $\mathbf{c'} = (r_{2m}, r_{2m-1}, \ldots, r_1)$.
This new polynomial thus produces a Reinhardt polygon in the same dihedral equivalence class as that of $f(z)$.
\end{proof}

For example, consider $n = 45 = 3 \cdot 3 \cdot 5$, with the partition $\mathbf{c} = (1,2)$.
We use
\[
A_1  =  \texttt{+00} \hspace{0.25in}
A_2  =  \texttt{0-0} \hspace{0.25in}
A_3  =  \texttt{0+0}
\]
and 
\[
B_1  =  \texttt{-+0} \hspace{0.25in}
B_2  =  \texttt{+0-} \hspace{0.25in}
B_3  =  \texttt{000} \hspace{0.25in}
B_4  =  \texttt{0+-} \hspace{0.25in}
B_5  =  \texttt{000}\,.
\]
Changing the first term of each $A_i$ and the last term of each $B_i$ 
    produces
\[
\widetilde A_1  =  \texttt{000} \hspace{0.25in}
\widetilde A_2  =  \texttt{+-0} \hspace{0.25in}
\widetilde A_3  =  \texttt{-+0}
\]
and 
\[
\widetilde B_1  =  \texttt{-+-} \hspace{0.25in}
\widetilde B_2  =  \texttt{+00} \hspace{0.25in}
\widetilde B_3  =  \texttt{00-} \hspace{0.25in}
\widetilde B_4  =  \texttt{0+0} \hspace{0.25in}
\widetilde B_5  =  \texttt{00-}\,.
\]
Reversing and changing the sign of each $\widetilde A_i$ and $\widetilde B_i$
    gives
\[
B_1'  =  \texttt{0-+} \hspace{0.25in}
B_2'  =  \texttt{0+-} \hspace{0.25in}
B_3'  =  \texttt{000} 
\]
and 
\[
A_1'  =  \texttt{+00} \hspace{0.25in}
A_2'  =  \texttt{0-0} \hspace{0.25in}
A_3'  =  \texttt{+00} \hspace{0.25in}
A_4'  =  \texttt{00-} \hspace{0.25in}
A_5'  =  \texttt{+-+}\,,
\]
which has the desired form.

\subsection{Remarks}\label{subsecRemarks}
First, we note that this method generalizes the construction employed in \cite{HareMossinghoff13} for creating Reinhardt polygons.
In effect, just one composition of $r$ was employed in that prior article, $\mathbf{c}=(1,r-1)$, so that each nonzero block of one of the polynomials in that construction was always \texttt{+-} or \texttt{-+}.
The rotated coefficient at the end of each $B_i$ was also required to be $0$.

Second, it may seem that we could generalize our construction further by employing more sub-blocks with an odd number of nonzero terms.
For example, when $i$ is even, we might allow selecting $B_{i,2}$ from $S_o(r_2+1,-1)$ and then $A_{i,3}$ from $S_o(r_3+1,1)$ (contrast with \eqref{eqnOverlaps}), in addition to what is allowed in our construction.
However, any polynomials we can create with this allowance can also be realized with the more restricted construction, by employing the overlapping position at the end of $B_{i,2}$ and the beginning of $A_{i,3}$ to change the parity in both sequences.
It follows that we may alter the parity of the number of nonzero values selected in both $A_{i,j}$ and $B_{i,j+1}$ (or both $A_{i,j}$ and $B_{i,j-1}$) without changing the the set of possible Reinhardt polynomials that may be constructed.
Since $F(z)$ must have an odd number of nonzero terms, we must specify that an odd number of nontrivial sub-blocks among the $A_{i,j}$ or $B_{i,j}$ are chosen to have an odd number of terms, for fixed $i$.
Due to the parity switches enabled by the overlaps between sub-blocks, we may assume without loss of generality that for fixed $i$ exactly one sub-block among the $A_{i,j}$ and $B_{i,j}$ is drawn with an odd number of terms, and we select the set $A_{i,1}$ in our construction.
The constraint $A_{1,1}(1)\neq0$ also prevents duplicate polynomials from being generated, due to the effect of the rotated element from the end of $B_q$.

\section{Counting Reinhardt polygons}\label{secCounting}

We determine the number of distinct Reinhardt polynomials for $n=pqr$ that may be produced by our construction, relative to a fixed composition of $r$ into an even number of parts.
We require one definition first: we say a sequence $v_1$, \ldots, $v_m$ is \textit{$d$-periodic} if $d\mid n$ and $v_k=-v_{k+d}$ for each index $k$ with $1\leq k\leq m-d$.

\begin{thm}\label{thmCountPolynomials}
Suppose $n=pqr$ with $p$ and $q$ distinct odd primes and $r\geq2$.
Let $\mathbf{c}=(r_1,\ldots,r_{2m})$ be a fixed composition of $r$, and let $r_e=\sum_{k=1}^m r_{2k}$ and $r_o=r-r_e$.
Then our construction produces $2^{r_o p + r_e q}$ distinct Reinhardt polynomials for $n$.
\end{thm}

\begin{proof}
We first show that each Reinhardt polynomial constructed in this way is 
    unique, for fixed composition $\mathbf{c} = (r_1, \ldots, r_{2m})$.
We then count the number of such Reinhardt polynomials.
Given $n=pqr$ and $\mathbf{c}=(r_1,\ldots r_{2m})$ as in the statement of the theorem.
Suppose that $A_1$, \ldots, $A_p$ and $B_1$, \ldots, $B_q$ are selected according to the construction using the initial sign $s$, and these give rise to the polynomials $f_1(z)$ and $f_2(z)$, respectively.
Let $g_1(z) = f_1(z)\Phi_q(-z^{pr})$ and $g_2(z) = f_2(z)\Phi_p(-z^{qr})$.
Similarly, suppose $A_1^*$, \ldots, $A_p^*$ and $B_1^*$, \ldots, $B_q^*$ are constructed with initial sign $s^*$, producing polynomials $g_1^*(z)$ and $g_2^*(z)$.
Suppose that $g_1(z)+g_2(z)=g_1^*(z)+g_2^*(z)$.
Then $g_1(z)-g_1^*(z)=g_2^*(z)-g_2(z)$, and the left side is $pr$-periodic and the right side is $qr$-periodic, so both sides are $r$-periodic.
It suffices then to investigate when $A_1-A_1^* = R(B_1^*-B_1)$.
For this, note first that many terms of $A_1$ and $A_1^*$ are $0$ (the ones in the sub-blocks of even index), so the corresponding positions must match $B_1$ and $B_1^*$.
Likewise, required zeros in $B_1$ and $B_1^*$ allow us to deduce matching elements in $A_1$ and $A_1^*$.
This leaves only the positions with index $1+\sum_{i=1}^k r_i$ with $0\leq k<2m$.
Let $C(k)$ denote the $k$th term of the sequence $C$.
If $s=s^*$, then $A_{1,1}(1)=A_{1,1}^*(1)=s$, and the alternating sign requirement, together with the parity constraint on $A_{1,1}$, ensures that $A_{1,1}(r_1+1)=A_{1,1}^*(r_1+1)$, and therefore $B_{1,2}(1)=B_{1,2}^*(1)$.
In the same way, we find that $B_{1,2}(r_2+1)=B_{1,2}^*(r_2+1)$, and thus $A_{1,3}(1)=A_{1,3}^*(1)$.
Continuing in this way, we find that $A_1=A_1^*$ and $B_1=B_1^*$.

If $s\neq s^*$, then the fact that $A_{1,1}$ and $A_{1,1}^*$ begin with opposite signs, followed by $r_1-1$ matching values, implies that these values must all be $0$, due to the alternating sign requirement.
By the parity condition, we then conclude that $A_{1,1}(r_1+1)=A_{1,1}^*(r_1+1)=0$ as well.
Then $B_{1,2}(1)=B_{1,2}^*(1)$, and these must be $0$ as well since $s\neq s^*$, and in the same way one finds that both $B_{1,2}$ and $B_{1,2}^*$ are entirely $0$.
Continuing, one deduces that $A_1(k)=B_1(k)=0$ for $2\leq k\leq r$, but then $A_1(1)-A_1^*(1)=2s$, while $-(B_r^*(1)-B_r(1))=0$, a contradiction.

It remains to count the number of different ways to construct a polynomial $F(z)$ for $n$, for a fixed composition $\mathbf{c}$.
Since $\abs{S_e(k,t)}=\abs{S_o(k,t)}=2^{k-1}$ for any positive integer $k$ and fixed $t\in\{-1,1\}$, there are $2^{r_j}$ ways to select $A_{i,j}$ when $j$ is odd, except for $A_{1,1}$, for which the number is $2^{r_1-1}$.
Then there are $2^{r_j}$ ways to select $B_{i,j}$ when $j$ is even, for fixed $s$.
Therefore, there are $2^{r_o-1}$ ways to construct $A_1$, then $2^{r_o}$ ways for each subsequent block $A_i$, and $2^{r_e}$ ways for each block $B_i$.
Since the blocks $A_1$, \ldots, $A_p$ and $B_1$, \ldots, $B_q$ are constructed independently, and there are two choices for $s$, the formula follows.
\end{proof}

For convenience, we say a Reinhardt polynomial is sporadic if it corresponds to a sporadic Reinhardt polygon.
We next determine a lower bound on the number of sporadic Reinhardt polynomials produced by our method for a fixed composition of $r$ into an even number of parts.

\begin{thm}\label{thmCountSporadicPolynomials}
Suppose $n=pqr$ with $p$ and $q$ distinct odd primes and $r\geq2$, let $\mathbf{c}=(r_1,\ldots,r_{2m})$ be a fixed composition of $r$, and let $r_e$ and $r_o$ be defined as above.
If $p$ and $q$ are the only odd prime divisors of $n$, then this construction produces exactly $2^r(2^{r_o(p-1)}-1)(2^{r_e(q-1)}-1)$ different sporadic Reinhardt polynomials.
Otherwise, let $t$ denote the smallest odd prime divisor of $n$ besides $p$ and $q$, and set $U = \lfloor2^{(2m pq + (r_o -m) p + (r_e-m) q)/t}\rfloor$.
Then this construction produces at least $2^r(2^{r_o(p-1)}-1)(2^{r_e(q-1)}-1)-U$ different sporadic Reinhardt polynomials.
\end{thm}

\begin{proof}
Suppose that $F(z)=\sum_{k=1}^n v_k z^{k-1}$ is a Reinhardt polynomial constructed by using this method, with $n=pqr$ and $\mathbf{c}$ as in the statement of the theorem.
Suppose further that $f_1(z)$ and $f_2(z)$ are constructed to build $F(z)$, with sequences $A_1$, \ldots, $A_p$ corresponding to $f_1(z)$ and sequences $B_1$, \ldots, $B_q$ for $f_2(z)$.
Let $g_1(z) = f_1(z)\Phi_q(-z^{pr})$ and $g_2(z) = f_2(z)\Phi_p(-z^{qr})$, so that $F(z)=g_1(z)+g_2(z)$.
Suppose further that $F(z)$ corresponds to a periodic Reinhardt polygon, so there exists a positive integer $d\mid n$ so that $v_k=-v_{k+d}$ for $1\leq k\leq n-d$.
Let $b=n/d$, and let $\ell$ be the number of nonzero $v_k$ with $1 \leq k \leq d$.
We see that there are $b\ell$ nonzero $v_k$ for $1 \leq k \leq n$ 
    by periodicity, and since this number is odd, then so are $b$ and $\ell$.
By replacing $d$ with an odd multiple of it if necessary, we may assume that 
    $b$ is an odd prime.
We consider three cases.

First, suppose $b=p$, so that $d=qr$.
Since $F(z)$ and $g_2(z)$ are both $qr$-periodic, it follows that $g_1(z)$ is as well.
Thus, for each $i$ with $1\leq i\leq p$, we have that $A_i=A_{i+2kq}$ and $A_i=-A_{i+(2k+1)q}$ for $1\leq k\leq(p-1)/2$, where the indices are taken modulo $p$ (using the residue system $\{1,\ldots,p\}$).
It follows that $A_i=(-1)^{i+j}A_j$ for $1\leq i\leq j\leq p$, and thus there are $2^{r_o}$ ways to construct $g_1(z)$, and thus $2^{r_o+r_e q}$ such polynomials $F(z)$.
Second, suppose $b=q$.
In the same way, we conclude that $B_i=(-1)^{i+j}B_j$ for $1\leq i\leq j\leq q$, and we find that there are $2^{r_e+r_o p}$ ways to construct $F(z)$ in this case.
It should be noted that it is possible for $F(z)$ to be both $pr$- and $qr$-periodic.
We will account for this later.

Third, suppose $b \not\in\{p,q\}$, so $r=b c$ for some positive integer $c$, and suppose further that $F(z)$ is neither $pr$-periodic nor $qr$-periodic.
Assume that $r_i \geq 2$ for some $i$.
Let $j$ be an index such that $v_j$ is in both $A_{k,i}$ and $B_{k,i}$ for some $k$, i.e., $j$ is not in an overlapping section.
We know that such a $j$ exists since $r_i \geq 2$.
Assume for now that $i$ is odd.
Then we see that $v_j \neq 0$ if and only if the corresponding term in $A_{k,i}$ is nonzero, 
    and further the corresponding term in $B_{k,i}$ is zero.
We then see that $v_j = (-1)^{k} v_{j+p r k} = (-1)^k v_{j + p b c k}$ for all $k$ by construction,
    as $g_1(z)$ is $pr$-periodic.
Further, by the assumption that $F(z)$ is $pqc$-periodic, we see that 
    $v_j = (-1)^{k} v_{j + p q c k}$ for all $k$.
As $\gcd(b, q) = 1$ we then see that $v_j = (-1)^{k} v_{j + p c k}$ for all $k$.
This adds considerable additional structure to a polynomial having this property.
In a similar fashion, if $i$ is even, we find that 
    $v_j = (-1)^{k} v_{j + q c k}$ for all $k$.

We may now bound the number of cases where $b \not\in \{p,q\}$.
Let $j$ be an index such that $v_j\neq0$ is in both $A_{k,i}$ and $B_{k,i}$ for some $k$
    with $i$ odd.
We see in a block of size $r$ that there are at most $r_o - m$ of these cases.
At most $p c/r \leq p/t$ of these blocks of size $r$ are needed to completely determine
    these $v_j$, since $b=r/c\geq t$.
Hence we have at most $2^{(r_o -m) p/t}$ choices for these
    $v_j$.
Similarly, we have at most $2^{(r_e-m) q/t}$ choices for the $v_j$ where the corresponding
    $i$ are even.
Last, for those nonzero $v_j$ with $j$ lying in an overlap, we observe that there are at most $2m$ such indices in a block of length $r$, and since the $v_j$ are $pqc$-periodic, we find that there are at most $2^{2m pq/t}$ choices for the $v_j$ in this category.
This produces the upper bound $U$ for the total number of cases where $b \not \in \{p,q\}$.

Since there are $2^r$ polynomials constructed that are both $pr$-periodic and $qr$-periodic, by using Theorem~\ref{thmCountPolynomials} we conclude that if $p$ and $q$ are the only odd prime divisors of $n$, then the number of polynomials that may be constructed by our method which correspond to sporadic Reinhardt polygons is precisely
\[
2^{r_o p + r_e q} - 2^{r_o+r_e q} - 2^{r_e+r_o p} + 2^r
= 2^r(2^{r_o(p-1)}-1)(2^{r_e(q-1)}-1),
\]
and otherwise we obtain the lower bound $2^r(2^{r_o(p-1)}-1)(2^{r_e(q-1)}-1) - U$.
\end{proof}

Next, we use this result to determine a lower bound for the number of sporadic Reinhardt $n$-gons which are produced by this construction for a fixed composition, and which are distinct under the dihedral action.
For this, we define the \textit{period} of a composition $\mathbf{c}$, denoted $\pi(\mathbf{c})$, as the smallest positive even integer $k$ for which $L^k(\mathbf{c})=\mathbf{c}$, where $L$ denotes the action of cyclically shifting the elements of a composition to the left by one position.

\begin{cor}\label{corBoundSporadicPolygons}
Suppose $n=pqr$ with $p$ and $q$ distinct odd primes, $r\geq2$, $\mathbf{c}$ is a fixed composition of $r$ into an even number of parts, and $r_e$, $r_o$, and $U$ are defined as above.
Set $U=0$ if $p$ and $q$ are the only odd prime divisors of $n$, and let $v = \sum_{i=1}^{\pi(\mathbf{c})} r_i$.
Then
\begin{equation}\label{eqnBoundSporadicPolygons}
E_1(n) \geq \frac{v}{r}\left(2^{r-2}\cdot\frac{2^{r_o(p-1)}-1}{p}\cdot\frac{2^{r_e(q-1)}-1}{q} - \frac{U}{4pq}\right).
\end{equation}
\end{cor}

\begin{proof}
Each of the polynomials from Theorem~\ref{thmCountSporadicPolynomials} corresponding to a sporadic Reinhardt polygon has $n$ different possible cyclic shifts, each of which is another Reinhardt polynomial, but shifting by $v$ positions to the left produces another polynomial in the family that we construct, so we must divide by $pq r/v$ to account for cyclic symmetries.
We must also divide by $2$ to account for the fact that $F(z)$ and $-F(z)$ correspond to the same Reinhardt polygon, and we divide by another factor of $2$ so that we are sure to place $F(z)$ and its reversal $z^n F(1/z)$ in the same equivalence class.
This accounts for flips in the dihedral action.
It is possible that some polynomials $F(z)$ that we produce have a palindromic coefficient pattern (such a polynomial is often called \textit{reciprocal}), but this only makes our lower bound on $E_1(n)$ more conservative.
The bound follows.
\end{proof}

We remark that the number of $(n/b)$-periodic polynomials that are
    neither $pr$-periodic nor $qr$-periodic is overestimated by the quantity $U$ of Theorem~\ref{thmCountSporadicPolynomials}.
For example, if $n = 3\cdot 5\cdot 7$ with $p = 5$, $q= 7$, and $r = t = 3$, 
    and $\mathbf{c} = (1,2)$, then our upper bound $U$ for the number of $35$-periodic polynomials constructible
    in this way is
    \[ U
         = \lfloor2^{(2 \cdot 5 \cdot 7 + (1-1) \cdot 5 + (2-1)\cdot 7)/3}\rfloor
         = \lfloor2^{77/3}\rfloor
         = 53\,264\,340. \]
The actual number of $35$-periodic polynomials constructible with these parameters is $0$.

It is possible for the number of $(n/b)$-periodic polynomials to be positive.
Let $n = 210 = 2 \cdot 3 \cdot 5 \cdot 7$ with $p = 3$, $q = 7$, 
    and $r = 10$, and let $\mathbf{c} = (1,1,1,1,1,1,1,1,1,1)$.
We select 
\[ 
A_1  =  \texttt{-0\,+-\,+-\,+-\,00} \hspace{0.25in}
A_2  =  \texttt{0+\,-+\,00\,00\,00} \hspace{0.25in}
A_3  =  \texttt{-0\,00\,00\,+-\,+-} 
\]
and 
\begin{equation*}
\begin{split}
B_1  &=  \texttt{0\,00\,00\,00\,+-\,0} \hspace{0.25in}
B_2  =  \texttt{0\,-+\,-+\,00\,00\,0} \hspace{0.25in}
B_3  =  \texttt{0\,00\,+-\,+-\,+-\,0} \\
B_4  &=  \texttt{0\,-+\,-+\,00\,-+\,-} \hspace{0.25in}
B_5  =  \texttt{+\,00\,+-\,+-\,+-\,0} \hspace{0.25in}
B_6  =  \texttt{0\,00\,00\,00\,-+\,-} \\
B_7  &=  \texttt{+\,00\,+-\,00\,00\,0}\,.
\end{split}
\end{equation*}
From these sequences we obtain a Reinhardt polynomial
    with period 42, which we designate by $f(z) \Phi_{5}(-z^{42})$, where the coefficients of $f(z)$
    are
  \[ \texttt{+0-+-+-0+00000-00000+00-+-000+-+000-+0-+-+}\,. \]

In addition, we remark that reciprocal polynomials do in fact arise in our method, although they are rather rare.
Figure~\ref{figRecip45s} exhibits the polygons for three of the $48$ essentially different reciprocal polynomials that may be constructed for $n=45$.
One may observe the evident reflective symmetry in each.

\begin{figure}[tbh]
\caption{Some sporadic Reinhardt polygons with $n=45$ sides having a reflective symmetry.}\label{figRecip45s}
\begin{tabular}{ccc}
\includegraphics[width=\rheinwidth]{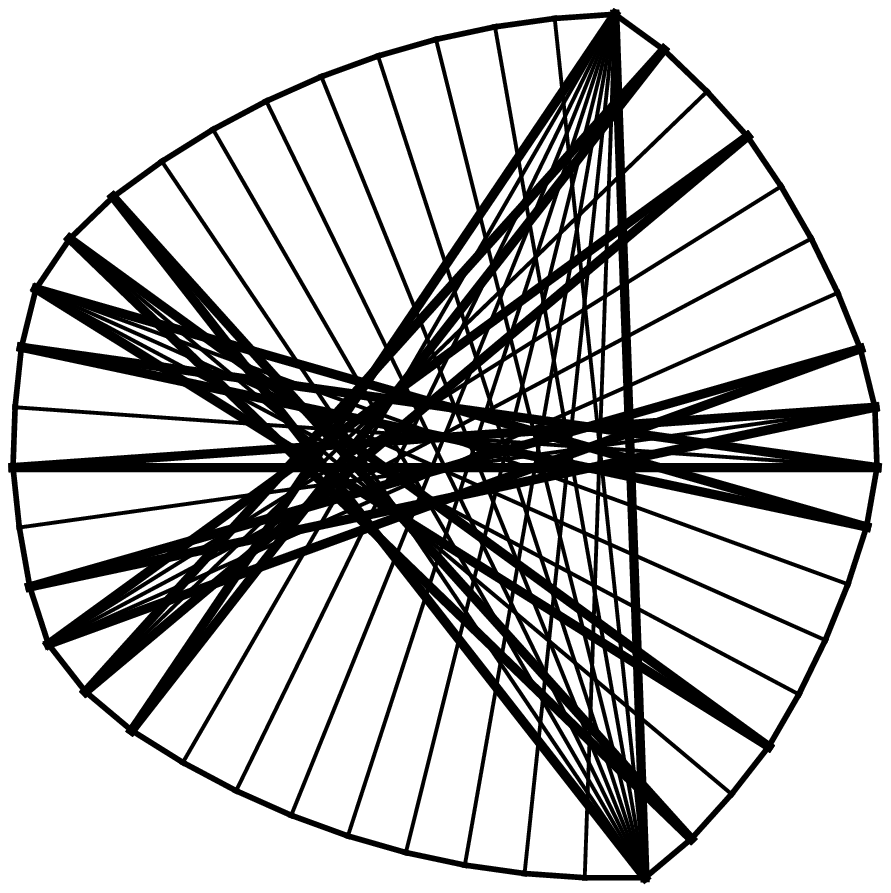} &
\includegraphics[width=\rheinwidth]{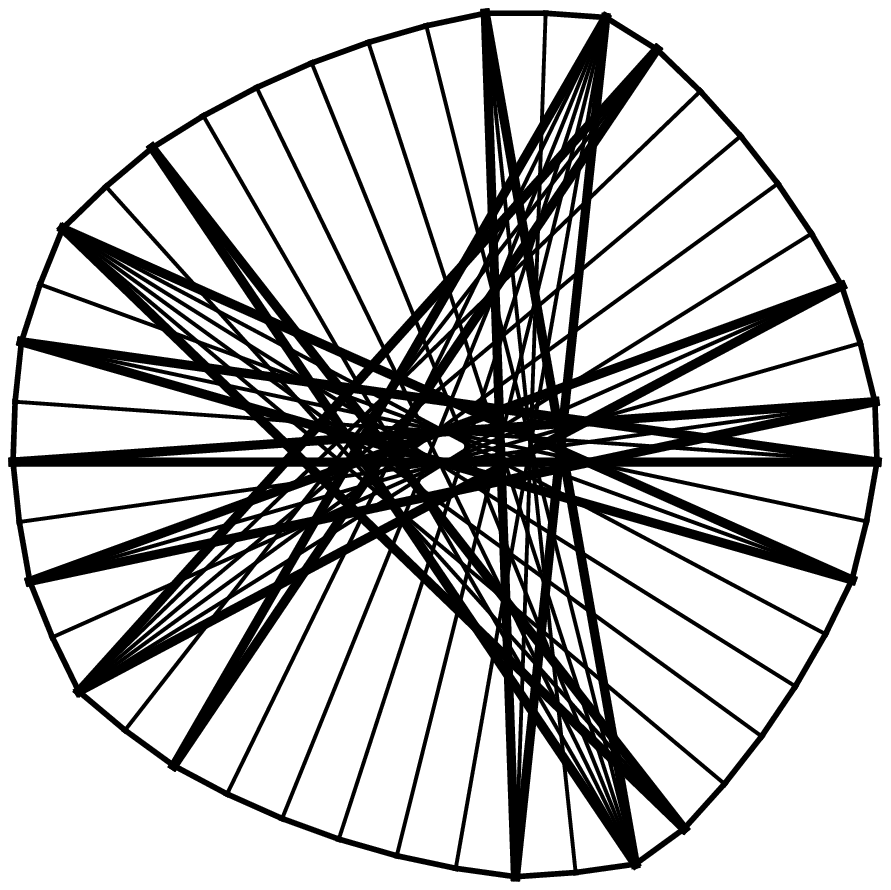} &
\includegraphics[width=\rheinwidth]{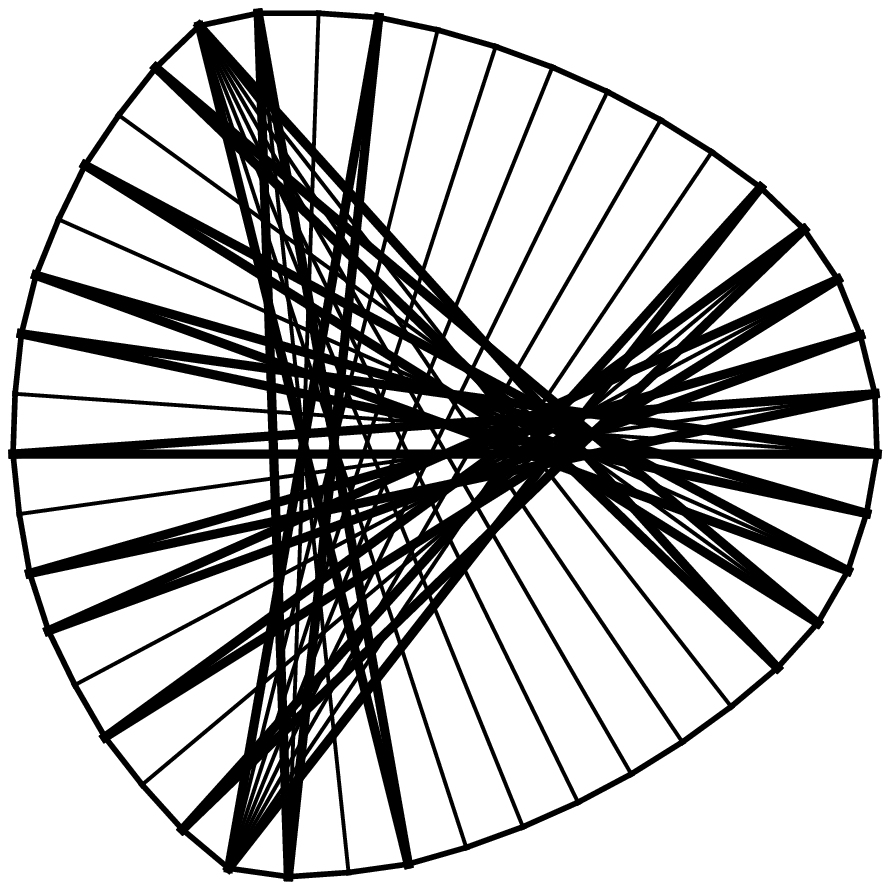}\\
\small $[2,1,1,4,1,2,1,1,9,$ &
\small $[2,2,2,5,2,1,6,2,$ &
\small $[2,1,1,1,2,1,2,1,1,7,1,2,$\\
\small \quad $\;9,1,1,2,1,4,1,1,2,1]$ &
\small \quad $\;2,6,1,2,5,2,2,2,1]$ &
\small \quad $\;2,1,7,1,1,2,1,2,1,1,1,2,1]$
\end{tabular}
\end{figure}

Next, we prove that there are almost always more sporadic Reinhardt polygons with $n$ sides than periodic ones.

\begin{proof}[Proof of Theorem~\ref{thmSporadicsWin}]
From the Prime Number Theorem, the density of positive integers $n$ having the form $n=2^a p^b$, with $p$ an odd prime, $a\geq1$, and $b\geq0$, or $n=2^a pq$, with $a\geq0$ and $p$ and $q$ distinct odd primes, is $0$.
We may assume then that $n$ has the form $n=pqr$, with $p$ and $q$ the two smallest distinct odd prime divisors of $n$, and that $r$ has at least one odd prime divisor.
Choose the composition $\mathbf{c}=(1,r-1)$ of $r$, so that $v = r$ in Corollary~\ref{corBoundSporadicPolygons}.
Since the smallest odd prime divisor of $r$ is clearly at least $3$, from \eqref{eqnBoundSporadicPolygons} have
\[
E_1(n) \geq\frac{2^r(2^{p-1}-1)(2^{(q-1)(r-1)}-1) - W}{4pq}
\]
where
\[
W = 2^{(2m pq + (r_o-m) p + (r_e-m) q)/3} = 2^{(2pq + (r-2)q)/3}.
\]
Next, since $W \ll 2^r(2^{p-1}-1)(2^{(q-1)(r-1)}-1)$ as $r\to\infty$, for fixed $p$ and $q$ and $n=pqr$ we find that
\[
E_1(n) \gg \frac{2^{n/p}}{4pq}\cdot (2^{p-1}-1)(2^{1-q}-2^{r(1-q)}),
\]
and using \eqref{eqnPeriodic}, we conclude that 
\[
\frac{E_1(n)}{E_0(n)} > \frac{r(2^{p-1}-1)}{p2^{q-1}}\left(1+o(1)\right)
\]
as $r$ grows large.
It follows that $E_1(n)>E_0(n)$ for all such $n$ with $r>cp2^{q-p}$, for some positive constant $c$.
We complete the proof by considering the density of positive integers $n=pqr$ having bounded gap between their two smallest odd prime divisors.

Clearly, the density of positive integers that are not divisible by a fixed prime $p$ is $1-1/p$, so the density of integers that have no odd prime divisor less than $m$ is
\[
\prod_{\substack{3\leq p \leq m\\p \textrm{\ prime}}} \left(1-\frac{1}{p}\right) = \frac{2+o(1)}{\log m}.
\]
It follows that the density of positive integers that are divisible by exactly one odd prime $p<m$ is
\[
\sum_{\substack{3\leq p\leq m\\p \textrm{\ prime}}}\frac{1}{p}\cdot\frac{1}{1-1/p}\cdot\frac{2}{\log m}\left(1+o(1)\right) = \frac{2\log\log m}{\log m}\left(1+o(1)\right).
\]
Thus, the density of integers having at least two distinct odd prime divisors less than $m$ is $1-\frac{2\log\log m}{\log m}(1+o(1))$, and any such integer $n$ has $q-p<m$ if $p$ and $q$ are the two smallest odd prime divisors of $n$.
We conclude that the density of positive integers $n=pqr$ whose two smallest odd prime divisors satisfy $q-p<m$ and for which $r>cm2^m$ approaches $1$ as $m\to\infty$.
\end{proof}

\section{Computations}\label{secComputations}

We implemented the construction described in Section~\ref{secConstruction} in \texttt{Maple}, and used it to create a large number of sporadic Reinhardt polygons for various values of $n$.
Given a qualifying integer $n$, then for each choice of distinct odd primes $p$ and $q$ dividing $n$, we set $r=n/pq$, and then for each composition of $r$ into an even number of parts, we compute all of the associated Reinhardt polynomials arising from our method.
Each polynomial constructed is normalized with respect to sign and with respect to its dihedral equivalence class, and all of the distinct Reinhardt polygons produced are recorded.

Let $\hat{E}_1(n)$ denote the number of sporadic Reinhardt polygons with $n$ sides up to dihedral equivalence that may be constructed by using the method of \cite{HareMossinghoff13}, and let $\ddot{E}_1(n)$ denote the number that can be built by using the more general method developed here.
In \cite{Mossinghoff11}, the exact value of $E_1(n)$ was computed for $24$ different values of $n$ that exhibit sporadic examples, and in \cite{HareMossinghoff13} it was shown that $\hat{E}_1(n)=E_1(n)$ for nineteen of these values.
Using our new method, we find that $\ddot{E}_1(n)=E_1(n)$ for the remaining five values.
It should be noted that $\ddot{E}_1(105)  \neq E_1(105)$, although the 
    exact value for $E_1(105)$ is not known.
The results for these five values are summarized in Table~\ref{tableCompleteCounts}.

\begin{table}[tbh]
\caption{Number of sporadic Reinhardt polygons constructed.}\label{tableCompleteCounts}
\begin{tabular}{cccccc}
$n$ & Factorization & $r$ & $\hat{E}_1(n)$ & $\ddot{E}_1(n)$ & $E_1(n)$\\\hline
$60$\fixspacetop & $2^2\cdot3\cdot5$ & $4$ & $3\,492$ & $4\,392$ & $4\,392$\\
$75$  & $3\cdot5^2$ & $5$ & $107\,400$ & $153\,660$ & $153\,660$\\
$84$  & $2^2\cdot3\cdot7$ & $4$ & $150\,444$ & $161\,028$ & $161\,028$\\
$90$  & $2\cdot3^2\cdot5$ & $6$& $3\,371\,568$ & $5\,385\,768$ & $5\,385\,768$\\
$140$ & $2^2\cdot5\cdot7$ & $4$ & $478\,548$ & $633\,528$ & $633\,528$
\end{tabular}
\end{table}

The nineteen other integers $n$ for which the value of $E_1(n)$ was computed precisely in \cite{HareMossinghoff13} all have the form $n=pqr$ with $r\in\{2,3\}$, and fourteen of these have $r=2$.
We compute $E_1(n)$ for five additional values of $n$ of the form $n=2pq$, and summarize these results in Table~\ref{tableSporadics2pq}.
The values newly computed here are marked with an asterisk.
For each of the nineteen integers $n$ recorded here, one may verify that $E_1(n)$ is the product of the Fermat quotients for $p$ and $q$ with base~$2$, as claimed in Theorem~\ref{thm2pq}.
This formula is proved in the next section.

\begin{table}[tbh]
\caption{Values of $E_1(2pq)$ (new values marked with $^*$).}\label{tableSporadics2pq}
\begin{tabular}{c|ccccccc}
$p$ & $q=5$ & $7$ & $11$ & $13$ & $17$ & $19$ & $23$\\\hline 
$3$      & 
$3$ & 
$9$ & 
$93$ & 
$315$ & 
$3\,855$ & 
$13\,797$ &  
$182\,361^*$ \\
$5$      &  & 
$27$ & 
$279 $ & 
$945 $ & 
$11\,565 $ & 
$41\,391 $ & 
$547\,083^*$ \\
$7$      & & & 
$837$ & 
$2\,835$ & 
$34\,695$ & 
$124\,173^*$ \\ 
$11$      & & & &
$29\,295$ &  
$358\,515^*$ &
$1\,283\,121^*$ &
\end{tabular}
\end{table}

We close this section by considering the case $n=105$, beginning with a proof of the second theorem announced in the introduction.

\begin{proof}[Proof of Theorem~\ref{thm105}]
In \cite{Mossinghoff11}, it was shown that the number of periodic Reinhardt $105$-gons is $E_0(105)=245\,518\,324$, and the method of \cite{HareMossinghoff13} showed that $E_1(105)\geq126\,714\,582$.
Using the current construction, we compute $\ddot{E}_1(105)=211\,752\,810$.
In addition, the methods of \cite{Mossinghoff11} allow us to count a number of additional qualifying polygons.
Let $E_1(n,m)$ denote the number of sporadic Reinhardt polygons with $n$ sides whose corresponding dihedral composition has largest part $m$.
The values of $E_1(105,m)$ for $m=2$ and for $m\geq12$ were computed in \cite{Mossinghoff11}, and the sum of these values is $12\,978\,294$.
Of these, we find that $6\,394\,732$ cannot be generated by the method of Section~\ref{secConstruction}.
We therefore need more than $27$ million additional sporadic polygons to show that $E_1(105)>E_0(105)$, and we find these by computing $E_1(105,m)$ for a few additional values of $m$.
By enumerating the dihedral compositions of $105$ with an odd number of parts and largest part $m$, and removing those with a periodic structure, we compute $E_1(105,11)=9\,194\,314$, $E_1(105,10)=15\,188\,197$, $E_1(105,9)=22\,135\,902$, and $E_1(105,8)=34\,641\,634$.
Of these $81\,160\,047$ polygons, $31\,449\,744$ cannot be generated by our construction.
We conclude that $E_1(105)\geq 249\,597\,286 > E_0(105)$.
Finally, there are only thirteen positive integers $n<105$ having the form $n=pqr$ with $p$ and $q$ distinct odd primes and $r\geq2$:  they are $n=30$, $42$, $45$, $60$, $63$, $66$, $70$, $75$, $78$, $84$, $90$, $99$, and $102$.
Table~1 of \cite{Mossinghoff11} verifies that $E_0(n)>E_1(n)$ for each of these values.
\end{proof}

Because our construction finds $59.8\%$ of all of the sporadic Reinhardt $105$-gons whose corresponding composition has restricted largest part in the ranges we considered, we might expect the true value of $E_1(105)$ to be over $350$ million.

Last, we remark that it may be possible to generalize our method for constructing Reinhardt polynomials by accounting for additional odd prime factors.
For example, let $F(z) = f_3(z) \Phi_3(-z^{35}) + f_5(z) \Phi_5(-z^{21}) + f_7(z) \Phi_7(-z^{15})$, where the coefficients of $f_3(z)$, $f_5(z)$, and $f_7(z)$ are
\begin{align*}
f_3(z) &: \texttt{0000000+-+0000-000+-+0000-+-0+0-00+},\\
f_5(z) &: \texttt{+000-000+-000++--+-+0},\\
f_7(z) &: \texttt{-000+00-00+00-0}.
\end{align*}
The coefficient sequence for $F(z)$ is then
\begin{align*}
\texttt{0}&\texttt{000000000+00000-+0-+-+00-+-+00-00+-+0-+0000-+-000+-+}\\
&\texttt{-+-+-0+-0+-00+00-+-+-0+-00+-+000-00+-00+-+-+-+000-0+}
\end{align*}
and so $F(z)$ is a Reinhardt polynomial for $n=105$.
However, $F(z)$ cannot be generated by using our algorithm.
To see this, suppose $p=3$ and $q=5$, and suppose that $F(z)=A(z) \Phi_3(-z^{35}) + B(z) \Phi_5(-z^{21})$, with $A(z)=\sum_{k=0}^{34} a_k z^k$ and $B(z)=\sum_{k=0}^{21} b_k z^k$.
By observing the constant term and the coefficients of $z^{14}$, $z^{42}$, and $z^{70}$ in $F(z)$ (respectively $0$, $0$, $0$, and $1$), we find $a_0 + b_0 = 0$, $a_7 + b_7 = 0$, $-a_7 + b_0 = 0$, and $a_0 - b_7 = 1$.
Subtracting the fourth equation from the first produces $b_0 + b_7 = -1$, but adding the second and third equations yields $b_0 + b_7 = 0$, a contradiction.
Similar analyses show that there are no decompositions using only $f_3(z)$ and $f_7(z)$, or using only $f_5(z)$ and $f_7(z)$.

\section{The case $n=2pq$}\label{sec2pq}

We show that our method constructs all Reinhardt polygons when $n=2pq$.
To establish this, we first require a generalization of \cite{HareMossinghoff13}*{Lemma~2.2}, which is the case $k=0$ of the following statement.

\begin{lem}\label{lemGenlDecomp}
Let $n = 2^k p q$ with $p$ and $q$ distinct odd primes and $k\geq0$, and suppose that $F(z)$ is a Reinhardt polynomial
for $n$.
Then there exist polynomials $f_1(z)$ and $f_2(z)$ with integer coefficients, $\deg(f_1) < 2^k p$, $\deg(f_2) < 2^k q$, and
\[
F(z) = f_1(z) \Phi_q(-z^{n/q}) + f_2(z) \Phi_p(-z^{n/p}).
\]
Further, we may choose $f_1(z)$ and $f_2(z)$ to have all their coefficients in $\{-1,0,1\}$.
\end{lem}

We omit the proof, as it is very similar to the one presented in \cite{HareMossinghoff13}, and may be obtained in essence by replacing each occurrence of $z$ when it appears as an argument of a cyclotomic polynomial in that proof with $z^{2^k}$.
(The roles of $p$ and $q$ have also been interchanged for convenience here.)

\begin{proof}[Proof of Theorem~\ref{thm2pq}]
Let $n=2pq$ with $p$ and $q$ distinct odd primes.
Clearly, $\mathbf{c}=(1,1)$ is the only possible composition of $r=2$ into an even number of parts. 
By Lemma \ref{lem:symmetry}, there are no additional sporadic
    Reinhardt polynomials introduced by reversing the roles of $p$ and $q$ 
    in the construction.
Hence Corollary~\ref{corBoundSporadicPolygons} produces $\ddot{E}_1(n)=(2^{p-1}-1)(2^{q-1}-1)/pq$.

We need to show that there are no additional sporadic Reinhardt polynomials that are not accounted for by our construction in Section \ref{secConstruction}, up to dihedral equivalence.
That is, we wish to show that $\ddot{E}_1(n) = E_1(n)$.
Let $F(z) = \sum_{i=0}^{2pq-1}u_i z^i$ be a sporadic Reinhardt polynomial for $n$.
By Lemma~\ref{lemGenlDecomp}, there exist polynomials $f_1(z)$ and $f_2(z)$ with all coefficients in $\{-1,0,1\}$, $\deg f_1 < 2p$, and $\deg f_2 < 2q$, such that $F(z) = f_1(z) \Phi_q(-z^{2p}) + f_2(z) \Phi_p(-z^{2q})$.
Write
\[
f_1(z)\Phi_q(-z^{2p}) = \sum_{i=0}^{2pq-1} s_i z^i
\textrm{\quad and\quad}
f_2(z)\Phi_p(-z^{2q}) = \sum_{j=0}^{2pq-1} t_j z^j,
\]
so that $s_{i+2p}=-s_i$ for $0\leq i<2p(q-1)$ and $t_{j+2q}=-t_j$ for $0\leq j<2(p-1)q$.
To show that $\ddot{E}_1(n) = E_1(n)$, we will establish the following seven statements.
\begin{enumerate}[(i)]
\item There exist $i \equiv 0$ mod $2$ and $j \equiv 1$ mod $2$ such that $u_i \neq 0$ and $u_j \neq 0$.
    \label{step:1}
\item There exist $i \equiv 0$ mod $2$ and $j \equiv 1$ mod $2$ such that $u_i = 0$ and $u_j = 0$.
    \label{step:2}
\item There exist $i \equiv 0$ mod $2$ and $i' \equiv 0$ mod $2$ such that $s_i \neq 0$ and $t_{i'} \neq 0$, 
      and similarly for $j \equiv 1$ mod $2$ and $j' \equiv 1$ mod $2$.
    \label{step:3}
\item There exist $\epsilon_1, \epsilon_2 \in \{-1, 1\}$ such that 
      \begin{equation}
      s_i  \in 
      \begin{cases} \{0, \epsilon_1 (-1)^{i/2}\}, &  \mathrm{if}\ i \equiv 0,\\
                            \{0, \epsilon_2 (-1)^{(i-1)/2}\}, &  \mathrm{if}\ i \equiv 1,
      \end{cases} 
      \label{eq:struct1}
      \end{equation}
      and 
      \begin{equation}
      t_{i'}  \in 
      \begin{cases} \{0, -\epsilon_1 (-1)^{{i'}/2}\}, &  \mathrm{if}\ {i'} \equiv 0,\\
                             \{0, -\epsilon_2 (-1)^{({i'}-1)/2}\}, &  \mathrm{if}\ {i'} \equiv 1.
      \end{cases}
      \label{eq:struct2}
      \end{equation}
    \label{step:4}
\item We may assume without loss of generality that $t_0 = t_1 = -1$.
    \label{step:5}
\item For each $k$, the pair $(s_{2k}, s_{2k+1}) \in (-1)^k \{$\texttt{0+}, \texttt{+0}$\}$.
      These are blocks of the form $S_o(2, (-1)^k)$, as required in \eqref{eqnAs}.
    \label{step:6}
\item For each $k$, the pair $(t_{2k+1}, t_{2k+2}) \in (-1)^k \{$\texttt{00}, \texttt{-+}$\}$.
      These are blocks of the form $S_e(2, (-1)^{k+1})$, as required in \eqref{eqnBs}.
    \label{step:7}
\end{enumerate}
This will complete the proof, as we will have shown that every sporadic Reinhardt polynomial for $n$ can be constructed by using the algorithm of Section \ref{secConstruction}.

For (\ref{step:1}), if $u_i = 0$ for all $i \equiv 0$ mod $2$, then $F(z) = z F_0(z^2)$ where $F_0(z)$ is a Reinhardt polynomial for $n = pq$.
By \cite[Theorem 1.1]{HareMossinghoff13}, $F_0(z)$ is periodic, and hence $F(z)$ is periodic, which contradicts our hypothesis.
Thus there exists an even index $i$ with $u_i \neq 0$.
The case for odd $j$ is similar.

For (\ref{step:2}), assume that $u_i \neq 0$ for all even $i$.
Then either $s_{0} =0$ and $t_{0} \neq 0$, or $s_{0} \neq 0$ and $t_{0} = 0$.
Assume that $s_{0} = 0$ and $t_{0} \neq 0$.
We then see that $t_{2qk} \neq 0$ for all $k$ by the periodicity of the $t_j$.
Since $u_{2qk} \neq 0$ for all $k$, we have that $s_{2qk} = 0$ for all $k$, and since the $s_i$ are $2p$-periodic and $\gcd(p,q) = 1$, we have that $s_{2k} = 0$ for all $k$.
As $f_1(z)$ is nontrivial, there exists an odd $j$ with $s_j \neq 0$.
Further, as $f_1(z) \Phi_q(-z^{2p})$ does not exhibit periodicity by $2q$ (since $F(z)$ is sporadic), there exist $j$ and $k$ such that $s_{j+2qk} \neq (-1)^k s_j$.
Consider the two sequences 
\[
(u_{j-1}, u_j, u_{j+1}) = (t_{j-1}, t_j + s_j, t_{j+1})
\]
and
\[
(u_{j+2kq-1}, u_{j+2kq}, u_{j+2kq+1}) = \left((-1)^k t_{j-1}, (-1)^k t_j + s_{j+2kq}, (-1)^k t_{j+1}\right).
\]
Since $u_i\neq0$ and $s_i=0$ for all even $i$, we know that $t_i\neq0$ for each even index $i$, and so both $t_{j-1}$ and $t_{j+1}$ are nonzero.
As $s_{j+2k} \neq (-1)^k s_j$, we see that these two sequences cannot both be of the correct form (i.e., alternating signs), producing a contradiction.
The roles of $s$ and $t$ are symmetric in this argument, so we reach the same conclusion in the other case, where $s_0 \neq 0$ and $t_0 = 0$.
Thus, there exists an even index $i$ where $u_i = 0$.
The case where $j \equiv 1$ mod $2$ is similar.

For (\ref{step:3}), we show that there exist even $i$ and $i'$ where $s_i \neq 0$ and $t_{i'} \neq 0$.
A similar argument handles the case where $j$ and $j'$ are both odd.
Using (\ref{step:2}), select an even index ${i_0}$ so that $u_{i_0} \neq 0$.
Then exactly one of $s_{i_0}$ and $t_{i_0}$ is nonzero.
Assume without loss of generality that $s_{i_0} \neq 0$ and $t_{i_0} = 0$.
If there also exists an even $i'$where $t_{i'} \neq 0$, we are done, so assume instead that $t_{i'} = 0$ for all even indices $i'$.
Consider the three sequences
\[
(u_{{i_0}+2p(k-1)}, \ldots, u_{{i_0}+2pk}), \;
(s_{{i_0}+2p(k-1)}, \ldots, s_{{i_0}+2pk}), \;
(t_{{i_0}+2p(k-1)}, \ldots, t_{{i_0}+2pk}).
\]
By noticing that $u_{i_0 + 2pk} = s_{i_0 + 2pk} + t_{i_0 + 2pk} = s_{i_0 + 2pk} = (-1)^k s_{i_0}$,
we have that the first sequence above starts and ends with nonzero terms of opposite sign.
Further, as $F(z)$ is a Reinhardt polynomial, the nonzero terms in the first sequence must have alternating sign.
The number of nonzero terms in the second sequence is constant for all $k$ by periodicity, so the number of nonzero terms contributed from the third sequence must have the same parity for all $k$ (either always even, or always odd).
Further, the $t_j$ can only be nonzero if $j$ is odd by assumption.
By the Chinese remainder theorem and the periodicity of the $t_j$, we then see that 
    $t_{i_0+2k}, \ldots, t_{i_0+2k+2p}$ has the same number of nonzero terms for all $k$.
Hence either all of the $t_j$ with $j$ odd are nonzero, or they are all zero.
They cannot all be zero, for otherwise $f_2(z)$ would be trivial. 
Hence $t_j \neq 0$ for each odd $j$.
From (\ref{step:2}), there exists an odd $j_0$ such that $u_{j_0} = 0$, so $s_{j_0} = - t_{j_0} \neq 0$.
The periodicity of the $s_i$ and $t_j$ then implies that $t_{j_0 + 2k} = (-1)^k t_{j_0}$ for each $k$, and
this produces that $f_2(z)$ is periodic with respect to $2p$.
Thus $F(z)$ is periodic, a contradiction.
Therefore, there exist even $i$ and $i'$ such that $s_i \neq 0$ and $t_{i'} \neq 0$.

Using this, we now show that there exist $\epsilon_1, \epsilon_2 \in \{-1, 1\}$ that 
    satisfy \eqref{eq:struct1} and \eqref{eq:struct2}.
Consider the case when $i \equiv i' \mod 2$ such that $s_i \neq 0$
     and $t_{i'} \neq 0$.
By the Chinese remainder theorem and the periodicity of the $s_i$ and $t_j$, we have that 
    $t_{i + 2k }  \in \{0, -(-1)^k s_i\}$ and $s_{i' + 2k} \in \{0, - (-1)^k t_{i'}\}$.
This establishes (\ref{step:4}).

Since $F(z)$ is a Reinhardt polynomial, it has an odd number of nonzero coefficients.
Thus, exactly one of $f_1(z)$ or $f_2(z)$ has an odd number of nonzero coefficients.
By reversing the roles of $p$ and $q$ if necessary, we may assume that $f_1(z)$ has an odd number of them, and $f_2(z)$ has an even number.
Note that $f_2(z)$ cannot have an equal number of zero and nonzero coefficients, since $q$ is odd.
If $f_2(z)$ has more zero coefficients than nonzero coefficients, consider the transformation
\begin{equation*}
s_i^* =
\begin{cases}
s_i -  \epsilon_1 (-1)^{i/2}, &  \mathrm{if}\ 2\mid i,\\
s_i -  \epsilon_2 (-1)^{(i-1)/2}, &  \mathrm{if}\ 2\nmid i,
\end{cases}
\mathrm{\ and\ \ }
t_j^* =
\begin{cases}
t_j +  \epsilon_1 (-1)^{j/2}, &  \mathrm{if}\ 2\mid j,\\
t_j +  \epsilon_2 (-1)^{(j-1)/2}, &  \mathrm{if}\ 2\nmid j.
\end{cases}
\end{equation*}
By \eqref{eq:struct1} and \eqref{eq:struct2}, this remains a valid decomposition of $F(z)$, $f_2^*(z)$ has an even number of nonzero coefficients, and $f_1^*(z)$ has an odd number.
In addition, $f_2^*(z)$ has more nonzero coefficients than zero coefficients.
Thus, we may assume without loss of generality that $f_2(z)$ has an even number of nonzero coefficients, and more nonzero coefficients than zero coefficients.
We now show that $f_2(z)$ has two adjacent nonzero coefficients with the same sign.
If there are three consecutive nonzero coefficients, then by \eqref{eq:struct1} and \eqref{eq:struct2} there must exist a \texttt{++} or \texttt{--} in the sequence for $f_2(z)$.
Otherwise, all subsequences of nonzero coefficients have length at most $2$, and so at least one subsequence of zero coefficients has length $1$.
Thus, there exist two subsequences of two nonzero coefficients whose starting indices have opposite parity, and again by \eqref{eq:struct1} and \eqref{eq:struct2} there will exist a \texttt{++} or \texttt{--}.
By cyclically shifting and multiplying by $-1$ if necessary, we may assume that $t_0 = t_1 = -1$.
This establishes (\ref{step:5}).

For (\ref{step:6}), note first that $(s_{2k},s_{2k+1}) \in (-1)^k \{$\texttt{0+}, \texttt{+0}, \texttt{++}$\}$, since no other subsequences can sum with $(-1)^k$\texttt{--} to produce a sequence from $\{-1, 0, 1\}$ whose nonzero terms alternate in sign.
We wish to show that no $(-1)^k$\texttt{++} pairs exist.
Assume that $(s_{2k}, s_{2k+1})$ has the form $(-1)^k$\texttt{++} for some $k$.
If this were the case for each $k$, then $f_1(z)$ would be $2q$-periodic, and hence $F(z)$ would not be sporadic.  
Hence there exists some $k$ such that $(s_{2k}, s_{2k+1}, s_{2k+2}, s_{2k+3}) = (-1)^k$\texttt{++0-} or $(-1)^k$\texttt{++-0}.
Since we are assuming that there is a $k$ with $(s_{2k}, s_{2k+1}) = (-1)^k$\texttt{++}, then by an argument similar to the one above, for each $\ell$ we have $(t_{2\ell}, t_{2\ell + 1}) \in (-1)^{\ell}\{$\texttt{--}, \texttt{0-}, \texttt{-0}$\}$.
We know that there is at least one $\ell$ such that 
    $(t_{2\ell}, t_{2\ell + 1}) = (-1)^\ell$\texttt{--},
so by a similar argument there exists some $\ell$ such that the sequence
    $(t_{2\ell}, t_{2\ell + 1}, t_{2\ell+2}, t_{2\ell + 3}) 
    = (-1)^\ell$\texttt{0-++} or $(-1)^\ell$\texttt{-0++}.
By the Chinese remainder theorem, we can find an index $m$ such that 
    $(t_{2m}, t_{2m + 1}, t_{2m+2}, t_{2m + 3}) = (-1)^m$ \texttt{0-++} or $(-1)^m$\texttt{-0++},
and 
    $(s_{2m}, s_{2m+1}, s_{2m+2}, s_{2m+3}) = (-1)^m$\texttt{++0-} or 
    $(-1)^m$\texttt{++0-}.
We see that none of these possibilities sums to a sequence of alternating signs, a contradiction.
Therefore $(s_{2k}, s_{2k+1}) \in (-1)^k \{$\texttt{0+}, \texttt{+0}$\}$, establishing (\ref{step:6}).

Since $f_1(z)$ is not trivial, we see there must exist at least one occurrence of each of $(-1)^k$\texttt{0+} and $(-1)^k$\texttt{0+}, for otherwise $f_1(z)$, and hence $F(z)$, would exhibit $2q$-periodicity.
Thus, we have $(s_{2k+1}, s_{2k+2}) \in (-1)^k \{$\texttt{+-}, \texttt{0-}, \texttt{+0}, \texttt{00}$\}$.
Moreover, as we range over $k$, we must witness at least one $(-1)^k$\texttt{+-}, at least one $(-1)^k$\texttt{00}, and at least one of $(-1)^k$\texttt{0-} or $(-1)^k$\texttt{+0} (possibly both).
Assume there exists a $k$ such that $(s_{2k+1}, s_{2k+2}) = (-1)^k$\texttt{0-}.
Consider the transformation 
\begin{equation*}
s_i^*  =
\begin{cases}
s_i, &  \mathrm{if}\ 2\mid i,\\
s_i -  \epsilon_2 (-1)^{(i-1)/2}, &  \mathrm{if}\ 2\nmid i,
\end{cases}
\mathrm{\ and\ \ }
t_j^*  =
\begin{cases}
t_j, &  \mathrm{if}\ 2\mid j,\\
t_j +  \epsilon_2 (-1)^{(j-1)/2}, &  \mathrm{if}\ 2\nmid j.
\end{cases}
\end{equation*}
Under this transformation, $f_1^*(z)$ will have an even number of nonzero coefficients, $f_2^*(z)$ will have an odd number, and we still have a valid decomposition of $F(z)$.
Furthermore, $f_1^*(z)$ contains $(s_{2k+1}^*, s_{2k+2}^*) = (-1)^k$\texttt{--}.
By the previous argument, all subsequences of the transformed $(t_{2k+1}^*, t_{2k+2}^*)$ have the form
    $(-1)^k$\texttt{0+} or $(-1)^k$\texttt{+0}.
Transforming back, we see that the original $(t_{2k+1}, t_{2k+2})$ has the form 
    $(-1)^k$\texttt{-+} or $(-1)^k$\texttt{00}.
These are the desired blocks of the form $S_e(2, (-1)^{k+1})$, and this completes (\ref{step:7}).
\end{proof}

\section*{Acknowledgements}

Computational support was provided in part by the Canadian Foundation for Innovation, the Ontario Research Fund, WestGrid, and the Centre for Interdisciplinary Research in Mathematics and Computer Science (IRMACS).
We also thank Robert Kenyon for his helpful comments on this manuscript.

\bibliographystyle{amsplain}

\begin{bibdiv}
\begin{biblist}

\bib{AudetHansenMessine09}{article}{
      author={Audet, C.},
      author={Hansen, P.},
      author={Messine, F.},
       title={Isoperimetric polygons of maximum width},
        date={2009},
        ISSN={0179-5376},
     journal={Discrete Comput. Geom.},
      volume={41},
      number={1},
       pages={45\ndash 60},
         url={http://dx.doi.org/10.1007/s00454-008-9103-9},
      review={\MR{2470069 (2009m:51031)}},
}

\bib{BezdekFodor00}{article}{
      author={Bezdek, A.},
      author={Fodor, F.},
       title={On convex polygons of maximal width},
        date={2000},
        ISSN={0003-889X},
     journal={Arch. Math. (Basel)},
      volume={74},
      number={1},
       pages={75\ndash 80},
         url={http://dx.doi.org/10.1007/PL00000413},
      review={\MR{1728365 (2000m:52025)}},
}

\bib{deBruijn53}{article}{
      author={de~Bruijn, N.~G.},
       title={On the factorization of cyclic groups},
        date={1953},
     journal={Nederl. Akad. Wetensch. Proc. Ser. A.},
      volume={15},
       pages={370\ndash 377},
}

\bib{Gashkov07}{article}{
      author={Gashkov, S.},
       title={Inequalities for convex polygons and {R}einhardt polygons},
        date={2007},
     journal={Mat. Prosveshchenye (3)},
      volume={11},
       pages={91\ndash 103},
        note={(Russian)},
}

\bib{HareMossinghoff13}{article}{
      author={Hare, K.~G.},
      author={Mossinghoff, M.~J.},
       title={Sporadic {R}einhardt polygons},
        date={2013},
        ISSN={0179-5376},
     journal={Discrete Comput. Geom.},
      volume={49},
      number={3},
       pages={540\ndash 557},
         url={http://dx.doi.org/10.1007/s00454-012-9479-4},
      review={\MR{3038529}},
}

\bib{Mossinghoff06AMM}{article}{
      author={Mossinghoff, M.~J.},
       title={A \$1 problem},
        date={2006},
        ISSN={0002-9890},
     journal={Amer. Math. Monthly},
      volume={113},
      number={5},
       pages={385\ndash 402},
         url={http://dx.doi.org/10.2307/27641947},
      review={\MR{2225472 (2006m:51021)}},
}

\bib{Mossinghoff06DCG}{article}{
      author={Mossinghoff, M.~J.},
       title={Isodiametric problems for polygons},
        date={2006},
        ISSN={0179-5376},
     journal={Discrete Comput. Geom.},
      volume={36},
      number={2},
       pages={363\ndash 379},
         url={http://dx.doi.org/10.1007/s00454-006-1238-y},
      review={\MR{2252109 (2007i:52014)}},
}

\bib{Mossinghoff11}{article}{
      author={Mossinghoff, M.~J.},
       title={Enumerating isodiametric and isoperimetric polygons},
        date={2011},
        ISSN={0097-3165},
     journal={J. Combin. Theory Ser. A},
      volume={118},
      number={6},
       pages={1801\ndash 1815},
         url={http://dx.doi.org/10.1016/j.jcta.2011.03.004},
      review={\MR{2793611}},
}

\bib{Reinhardt22}{article}{
      author={Reinhardt, K.},
       title={{E}xtremale {P}olygone gegebenen {D}urchmessers},
        date={1922},
     journal={Jahresber. Deutsch. Math.-Verein.},
      volume={31},
       pages={251\ndash 270},
}

\end{biblist}
\end{bibdiv}
\def\lfhook#1{\setbox0=\hbox{#1}{\ooalign{\hidewidth
  \lower1.5ex\hbox{'}\hidewidth\crcr\unhbox0}}}

\end{document}